\documentclass[12pt]{amsart}

\usepackage[latin1]{inputenc}
\usepackage{subfigure}
\usepackage{comment}
\usepackage[margin=2.3cm]{geometry}
\usepackage{cancel}
\usepackage{wrapfig}
\usepackage{ifthen}

\usepackage{amsmath}
\usepackage{ytableau}

\newtheorem{theorem}{Theorem}[section]
\newtheorem{conjecture}[theorem]{Conjecture}
\newtheorem{problem}[theorem]{Open problem}
\newtheorem*{metaconjecture}{Meta-conjecture}
\newtheorem{proposition}[theorem]{Proposition}
\newtheorem{lemma}[theorem]{Lemma}

\newtheorem {corollary}[theorem]{Corollary}
\theoremstyle{definition}
\newtheorem{definition}[theorem]{Definition}
\newtheorem{example}[theorem]{Example}
\newtheorem{remark}[theorem]{Remark}
\newtheorem{notation}[theorem]{Notation}

\allowdisplaybreaks

\def\N{\mathbb {N}}
\renewcommand{\skew}{\lambda/\mu}
\newcommand{\p}[1]{\mathcal #1} 
\newcommand{\set}[1]{\{ #1 \}}
\newcommand{\mf}[1]{\mathfrak{#1}}

\newcommand{\smallbinom}[2]{{\left(\begin{smallmatrix}#1 \\ #2\end{smallmatrix}\right)}}
\newcommand{\three}[3]{\genfrac{}{}{0pt}{3}{\genfrac{}{}{0pt}{3}{#1}{#2}}{#3}}
\DeclareMathOperator{\md}{mod}
\DeclareMathOperator{\dv}{div}
\DeclareMathOperator{\hgt}{ht}
\DeclareMathOperator{\Par}{Par}

\author[M.\ Konvalinka]{Matja\v z Konvalinka}\thanks{The author was partially supported by Research Program L1-­069 of the Slovenian Research Agency}

\title[Residue and quotient tables]{The role of residue and quotient tables in the theory of $k$-Schur functions}

\address{Department of Mathematics, University of
Ljubljana, Slovenia}

\begin{document}

\begin{abstract}
 Recently, residue and quotient tables were defined by Fishel and the author, and were used to describe strong covers in the lattice of $k$-bounded partitions. In this paper, we show or conjecture that residue and quotient tables can be used to describe many other results in the theory of $k$-bounded partitions and $k$-Schur functions, including $k$-conjugates, weak horizontal and vertical strips, and the Murnaghan-Nakayama rule. Evidence is presented for the claim that one of the most important open questions in the theory of $k$-Schur functions, a general rule that would describe their product, can be also concisely stated in terms of residue tables.
\end{abstract}

\maketitle

\section{Introduction} \label{intro}

In 1988, Macdonald \cite{Mac95} introduced a new class of polynomials, now called \emph{Macdonald polynomials}, and conjectured that they expand positively in terms of Schur functions. This conjecture, verified in \cite{haiman}, has led to an enormous amount of work, including the development of $k$-Schur functions, defined first in \cite{llm03}; Lapointe, Lascoux, and Morse conjectured that they form a basis of a certain subspace of the space of symmetric functions and that the Macdonald polynomials indexed by partitions whose first part is not larger than $k$ expand positively in terms of $k$-Schur functions, leading to a refinement of the Macdonald conjecture. Since then, $k$-Schur functions have been found to arise in other contexts; for example, as Schubert classes in the quantum cohomology of the Grassmannian \cite{lm08}, and, more generally, in the cohomology of the affine Grassmannian \cite{lam06}.

\medskip

It turns out that $k$-Schur functions are, both technically and conceptually, a very difficult generalization of Schur functions, with many major questions either unanswered, or only conjecturally resolved; for example, there are several different and only conjecturally equivalent definitions of $k$-Schur functions (\textit{e.g.}\ the definition via atoms from \cite{llm03} and the definition via strong marked tableaux from \cite{LLMS10}). Probably the most important open problem is to find a Littlewood-Richardson rule for $k$-Schur functions, \textit{i.e.}\ a general rule for the expansion of the product of two $k$-Schur functions in terms of $k$-Schur functions.

\medskip

It is known that $k$-Schur functions (at $t = 1$) and Fomin-Gelfand-Postnikov quantum Schubert polynomials can be obtained from each other by a rational substitution (see \cite{ls}). Therefore a multiplication rule for $k$-Schur functions would also imply a multiplication rule for quantum Schubert polynomials. See also \cite[\S 2.2.5]{llmssz}.

\medskip

Recently, new tools, \emph{residue} and \emph{quotient tables}, were introduced, and it was shown that strong marked covers can be elegantly expressed in terms of them. See \cite[Theorem 5.2]{fk} and Theorem \ref{strong}. This paper hopes to convince the reader that residue and quotient tables (defined in Section \ref{residue} in a way that is slightly different than in \cite{fk}) are extremely useful in the theory of cores, $k$-bounded partitions, and $k$-Schur functions, and that they have the potential to solve many open questions, including the $k$-Littlewood-Richardson rule in full generality.

\medskip

As motivation, let us present two examples that illustrate the power of these tables.

\begin{example}
 Say that we are given the $4$-bounded partition $\lambda = 44432211111$ and we want to find all $4$-bounded partitions covered by $\lambda$ (the definitions are given in the next section). The residue table for this partition (the definition is given in Section \ref{residue}) is
 $$\begin{matrix}
 1 & 0 & 1 & 0 \\
   & 2 & 1 & 0 \\
   &   & 1 & 0 \\
   &   &   & 0
\end{matrix}$$
 According to Theorem \ref{strong}, each entry in the residue table that is strictly smaller than all the entries to its left gives us a strong cover (possibly with multiplicity $0$). Such an entry is, for example, $1$ in position $(2,3)$. This tells us that $44432211111$ covers $44441111111$ in the strong order (with multiplicity $1$, which can be computed from the quotient table).
\end{example}

This is a stunningly simple way to compute strong covers. The power of this description was recently illustrated by Lapointe and Morse, who needed it to reprove the Monk's formula for quantum Schubert polynomials; see Subsection \ref{almost}.

\medskip

The following example hints that there could be a $k$-Littlewood-Richardson rule expressible in terms of residue tables.

\begin{example} \label{exintro}
 Say that we want to compute the coefficient of $s_{\lambda \cup n}^{(k)}$ (here $\lambda \cup n$ means that we add the part $n$ to $\lambda$) in the product $s_\lambda^{(k)} s_{n-2,2}^{(k)}$ (for $n \leq k$) with $\lambda$ a $k$-bounded partition. For example, take $n = 6$. Then, according to Section \ref{kLR}, we have the following $9$ sets of conditions:
 \begin{description}
  \item[C1] $13,14,15,16,23,24,25,26$
  \item[C2] $12,14,15,16,34,35,36$
  \item[C3] $12,13,15,16,45,46$
  \item[C4] $12,13,14,56$
  \item[C5] $24,25,26,34,35,36$
  \item[C6] $23,25,26,45,46$
  \item[C7] $23,24,56$
  \item[C8] $35,36,45,46$
  \item[C9] $34,56$
 \end{description}
 The coefficient of $s_{\lambda \cup 6}^{(k)}$ in the product $s_\lambda^{(k)} s_{4,2}^{(k)}$ is (conjecturally) equal to the number of $i$, $1 \leq i \leq 9$ for which \textbf{Ci} is satisfied for the residue table $R = (r_{ij})$ of $\lambda$. Here, a condition $IJ$ for $R$ is interpreted as $r_{I6}\neq r_{J6}$. So written out in full, \textbf{C6} is
 $$r_{26} \neq r_{36} \mbox{ and } r_{26} \neq r_{56} \mbox{ and } r_{26} \neq r_{66} \mbox{ and } r_{46} \neq r_{56} \mbox{ and } r_{46} \neq r_{66}.$$
 The reader can check that the coefficient of $s_{655554442}^{(10)}$ in $s_{55554442}^{(10)}s_{42}^{(10)}$ is $4$, owing to the fact that precisely the conditions \textbf{C1}, \textbf{C6}, \textbf{C7}, \textbf{C8} are satisfied for the residue table
 $$\begin{smallmatrix}
 0 & 1 & 1 & 4 & 2 & 2 & 2 & 2 & 0 & 0 \\
   & 1 & 1 & 4 & 2 & 2 & 2 & 2 & 0 & 0 \\
   &   & 0 & 3 & 1 & 1 & 1 & 1 & 1 & 0 \\
   &   &   & 3 & 1 & 1 & 1 & 1 & 1 & 0 \\
   &   &   &   & 4 & 4 & 0 & 0 & 0 & 0 \\
   &   &   &   &   & 0 & 0 & 0 & 0 & 0 \\
   &   &   &   &   &   & 0 & 0 & 0 & 0 \\
   &   &   &   &   &   &   & 0 & 0 & 0 \\
   &   &   &   &   &   &   &   & 0 & 0 \\
   &   &   &   &   &   &   &   &   & 0 \\
\end{smallmatrix}
$$
of the $10$-bounded partition $\lambda = 55554442$.
\end{example}
Such conjectures were checked for all $k$-bounded partitions $\lambda$ for several $n$ and for many $k$. 

\medskip

Anybody who has studied $k$-Schur functions (and quantum Schubert polynomials) will agree that it is quite amazing that such simple conditions exist. Note that the conditions do not even contain $k$ explicitly (the parameter is, of course, implicit in the definition of the residue table and its flip).

\medskip

This paper is organized as follows. In Section \ref{cores}, we present one of the possible definitions of $k$-Schur functions (with parameter $t$ equal to $1$), known to be equivalent to the definition via strong marked tableaux from \cite{LLMS10}. In Section \ref{residue}, we define the {residue} and {quotient tables} of a $k$-bounded partition and list their (possible) applications. In Section \ref{s:kconj}, we present their geometric meaning, show how to compute the $k$-conjugate of a $k$-bounded partition directly in terms of the quotient table (without resorting to cores), and how to compute the size of the corresponding core. In Section \ref{strips}, we describe strong and weak covers, weak horizontal strips, and weak vertical strips in terms of residue tables. In Section \ref{application}, we show how to use this new description to prove a known multiplication result in an easier way. In Section \ref{mn}, we present a restatement of a (special case of a) known result, the Murnaghan-Nakayama rule for $k$-Schur functions, first proved in \cite{bsz}. Our version (in terms of residue tables, of course) is simpler and should serve as one of the most convincing proof of the power of the new tools. We continue with Section \ref{further}, in which we explore several other directions where residue and quotient tables could be useful. In Section \ref{kLR}, we present some conjectures about the multiplication of $k$-Schur functions. The conjectures indicate that there could be a general Littlewood-Richardson rule for $k$-Schur functions involving residue tables. Some of the technical proofs are deferred to Section \ref{proofs}.

\medskip

A reader who wishes to get a basic idea of the paper (and already knows some $k$-Schur theory) should:
\begin{itemize}
 \item read and absorb Notation \ref{notation} and Remark \ref{remark} on \ifthenelse{\equal{\pageref{notation}}{\pageref{remark}}}{page \pageref{notation}}{pages \pageref{notation} and \pageref{remark}};
 \item read Section \ref{residue};
 \item skim through Corollary \ref{kconj}, Theorem \ref{size}, Theorem \ref{strong}, Theorem \ref{whs}, Theorem \ref{wvs}, Conjecture \ref{mn1}, Conjecture \ref{mn2}, Proposition \ref{jenn}, and the corresponding examples.
\end{itemize}

\section{Cores, $k$-bounded partitions, Schur and $k$-Schur functions} \label{cores}

\subsection{Basic terminology}

A \emph{partition} is a sequence $\lambda=(\lambda_1,\ldots,\lambda_\ell)$ of weakly decreasing positive integers, called the \emph{parts} of $\lambda$. The \emph{length} of $\lambda$, $\ell(\lambda)$, is the number of parts, and the \emph{size} of $\lambda$, $|\lambda|$, is the sum of parts; write $\lambda \vdash n$ if $|\lambda| = n$, and denote by $\Par(n)$ the set of all partitions of size $n$. Denote by $m_j(\lambda)$ the number of parts of $\lambda$ equal to $j$. The \emph{Young diagram} of a partition $\lambda$ is the left-justified array of cells with $\ell(\lambda)$ rows and $\lambda_i$ cells in row $i$. (Note that we are using the English convention for drawing diagrams.) We will often refer to both the partition and the diagram of the partition by $\lambda$.  If $\lambda$ and $\mu$ are partitions, we write $\lambda \cup \mu$ for the partition satisfying $m_j(\lambda \cup \mu) = m_j(\lambda) + m_j(\mu)$ for all $j$. We write $\mu \subseteq \lambda$ if the diagram of $\mu$ is contained in the diagram of $\lambda$, \textit{i.e.}\ if $\ell(\mu)\leq\ell(\lambda)$ and $\mu_i\leq\lambda_i$ for $1\leq i\leq\ell(\mu)$. If $\mu \subseteq \lambda$, we can define the \emph{skew diagram} $\skew$ as the cells which are in the diagram of $\lambda$ but not in the diagram of $\mu$. If $\lambda$ and $\mu$ are partitions of the same size, we say that $\mu \leq \lambda$ in the \emph{dominance order} (or that $\lambda$ \emph{dominates} $\mu$) if $\mu_1 + \cdots + \mu_i \leq \lambda_1 + \cdots + \lambda_i$ for all $i$. If no two cells of $\skew$ are in the same column (resp., row), we say that $\skew$ is a \emph{horizontal} (resp., \emph{vertical}) \emph{strip}.  If the skew shape $\lambda/\mu$ is connected and contains no $2\times 2$ block, we call it a \emph{ribbon}; if it contains no $2 \times 2$ block (and is not necessarily connected), it is a \emph{broken ribbon}. (Note that in \cite{bsz}, broken ribbons are called ribbons, and ribbons are called connected ribbons.) The \emph{height} $\hgt(\lambda/\mu)$ of a ribbon is the number of rows it occupies, minus $1$, and the height of a broken ribbon is the sum of the heights of its components.

\medskip

For $1\leq i\leq \ell(\lambda)$ and $1\leq j\leq \lambda_i$, cell $(i,j)$ refers to the cell in row $i$, column $j$ of $\lambda$. The \emph{conjugate} of $\lambda$ is the partition $\lambda'$ whose diagram is obtained by reflecting the diagram of $\lambda$ about the diagonal. The \emph{$(i,j)$-hook} of a partition $\lambda$ consists of the cell $(i,j)$ of $\lambda$, all the cells to the right of it in row $i$, together with all the cells below it in column $j$. The \emph{hook length} $h_{ij}^\lambda$ is the number of cells in the $(i,j)$-hook, $h_{ij}^\lambda = \lambda_i + \lambda'_j-i-j+1$.

\medskip

Let $n$ be a positive integer. A partition $\pi$ is an \emph{$n$-core} if $h_{ij}^\pi \neq n$ for all $(i,j) \in \pi$. There is a close connection between $(k+1)$-cores and \emph{$k$-bounded partitions}, which are partitions whose parts are at most $k$ (equivalently, $\lambda = \emptyset$ or $\lambda_1 \leq k$). Indeed, in \cite{LLM05}, a simple bijection between $(k+1)$-cores and $k$-bounded partitions is presented. Given a $(k+1)$-core $\pi$, let $\lambda_i$ be the number of cells in row $i$ of $\pi$ with hook-length $\leq k$. The resulting $\lambda = (\lambda_1,\ldots,\lambda_\ell)$ is a $k$-bounded partition, we denote it by $\mf b^{(k)}(\pi)$. Conversely, given a $k$-bounded partition $\lambda$, move from the last row of $\lambda$ upwards, and in row $i$, shift the $\lambda_i$ cells of the diagram of $\lambda$ to the right until their hook-lengths are at most $k$. The resulting $(k+1)$-core is denoted by $\mf c^{(k)}(\lambda)$. For a $k$-bounded partition $\lambda$, call $\mf b^{(k)}(\mf c^{(k)}(\lambda)')$ the \emph{$k$-conjugate of $\lambda$} and denote it by $\lambda^{(k)}$. Denote the set of all $k$-bounded partitions of size $n$ by $\Par(n,k)$.

\begin{example} \label{ex:core}
 On the left-hand side of Figure \ref{fig1}, the hook-lengths of the cells of the $5$-core $\pi = 9 5 3 2 1 1$ are shown, with the ones that are $\leq 4$ underlined. That means that $\mf b^{(4)}(\pi) = 432211$.
 \ytableausetup{centertableaux}

\begin{figure}[!ht] 
\begin{center}
 \ytableausetup{boxsize=1.2em}
 \begin{ytableau}
14 & 11 & 9 & 7 &  6 & \underline 4 & \underline 3 & \underline 2 &
\underline 1 & \none & \none & \none & \none & \none & & & & & \\
9 & 6 & \underline 4 & \underline 2 & \underline 1 & \none & \none & \none &
\none & \none & \none & \none & \none  & & & & \\
6 & \underline 3 & \underline 1 & \none & \none & \none & \none& \none & \none
& \none & \none & \none & & \\
\underline 4 & \underline 1  & \none & \none & \none & \none& \none & \none &
\none & \none & \none & \none & &\\
\underline 2  & \none & \none & \none & \none & \none & \none& \none & \none &
\none & \none & \none &\\
\underline 1
\end{ytableau}

\end{center}
\caption{Bijections $\mf b^{(k)}$ and $\mf c^{(k)}$.}\label{fig1}
\end{figure}

 The right-hand side shows the construction of $\mf c^{(6)}(\lambda) = 7 5 2 2 1$ for the $6$-bounded partition $\lambda = 54221$. It follows that $54221^{(6)} = 3322211$.
\end{example}

Of particular importance are $k$-bounded partitions $\lambda$ that satisfy $m_j(\lambda) \leq k-j$ for all $j=1,\ldots,k$. We call such partitions \emph{$k$-irreducible partitions}, see \cite{llm03}. The number of $k$-irreducible partitions is clearly $k!$.

\subsection{Schur functions}

A \emph{weak composition} $\alpha = (\alpha_1,\alpha_2,\ldots)$ is a sequence of nonnegative integers, all but finitely many of them $0$; we let $|\alpha| = \sum_i \alpha_i$ denote its \emph{size}. For commutative variables $x_1,x_2,\ldots$ and a weak composition $\alpha = (\alpha_1,\alpha_2,\ldots)$, write $x^\alpha$ for $x_1^{\alpha_1}x_2^{\alpha_2} \cdots$. A \emph{homogeneous symmetric function of degree $n$} over a commutative ring $R$ with identity is a formal power series $\sum_\alpha c_\alpha x^\alpha$, where the sum ranges over all weak compositions $\alpha$ of size $n$, $c_\alpha$ is an element of $R$ for every $\alpha$, and $c_\alpha = c_\beta$ if $\beta$ is a permutation of $\alpha$. A symmetric function is a finite sum of homogeneous symmetric functions (of arbitrary degrees). Let $\Lambda^n$ denote the (finite-dimensional) vector space of symmetric functions of degree $n$ and let $\Lambda$ denote the algebra of symmetric functions (with natural operations).

\medskip

The vector space $\Lambda^n$ has several natural bases. For a partition $\lambda$, define the \emph{monomial symmetric function} $m_\lambda$ by $\sum_\alpha x^\alpha$, where the sum is over all distinct permutations $\alpha$ of $\lambda$. Define the \emph{elementary symmetric function} $e_\lambda$ as $e_{\lambda_1}\cdots e_{\lambda_\ell}$, where $e_n = m_{1^n}$. Define the \emph{complete homogeneous symmetric function} $h_\lambda$ as $h_{\lambda_1} \cdots h_{\lambda_\ell}$, where $h_n = \sum_{\tau \vdash n} m_\tau$. Define the \emph{power sum symmetric function} $p_\lambda$ as $p_{\lambda_1} \cdots p_{\lambda_\ell}$, where $p_n = m_n$. The earliest results in the theory of symmetric functions show that $\{b_\lambda: \lambda \vdash n\}$ is a basis of $\Lambda^n$, where $b$ stands for either $m$, $e$, $h$ or $p$.

\medskip

Define a \emph{semistandard Young tableau} $T$ of \emph{shape} $\lambda$ as a filling of the Young diagram of $\lambda$ with positive integers such that the entries are weakly increasing in each row and strictly increasing in each column. If the tableau $T$ has $\mu_j$ copies of the integer $j$, we call $\mu$ the \emph{weight} of $T$. In other words, a semistandard Young tableau of shape $\lambda$ and weight $\mu$ is a sequence of partitions $\lambda^{0} \subseteq \lambda^{1} \subseteq \ldots \subseteq \lambda^{m}$ such that $\lambda^{0} = \emptyset$, $\lambda^{m} = \lambda$, and $\lambda^{i}/\lambda^{i-1}$ is a horizontal strip of size $\mu_i$. For partitions $\lambda$ and $\mu$ (of the same size), define the \emph{Kostka number} $K_{\lambda\mu}$ as the number of semistandard Young tableaux of shape $\lambda$ and weight $\mu$. It is easy to see that $K_{\lambda \lambda} = 1$ and $K_{\lambda \mu} = 0$ unless $\lambda \geq \mu$. In other words, the matrix $(K_{\lambda\mu})_{\lambda,\mu \in Par(n)}$ is upper-triangular with $1$'s on the diagonal (in any linear extension of the dominance order) and hence invertible. Therefore we can define \emph{Schur functions} by
$$h_\mu = \sum_\lambda K_{\lambda\mu} s_\lambda.$$
The set $\set{s_\lambda \colon \lambda \vdash n}$ forms the most important basis of $\Lambda^n$. If $\mu \subseteq \lambda$, we can analogously define a semistandard Young tableau of shape $\skew$ and the \emph{skew Schur function} $s_{\skew}$.

\medskip

The \emph{Pieri rule} and the \emph{conjugate Pieri rule} state that
$$s_\lambda s_n = s_\lambda h_n = \sum_\nu s_\nu, \qquad s_\lambda s_{1^n} = s_\lambda e_n = \sum_\nu s_\nu$$
where the first (resp., second) sum is over all $\nu$ for which $\nu/\lambda$ is a horizontal (resp., vertical) strip of size $n$.

\medskip

The \emph{Murnaghan-Nakayama rule} states that
$$s_\lambda p_n = \sum_\nu (-1)^{\hgt(\nu/\lambda)} s_\nu,$$
where the sum is over all $\nu$ for which $\nu/\lambda$ is a ribbon of size $n$.

\begin{notation} 
 For a set $S \subseteq \set{1,2,\ldots}$ and a partition $\lambda$, denote by $\lambda^S$ the result of adding a cell to $\lambda$ in columns determined by $S$. In other words, 
$$m_j(\lambda^S) = \left\{ 
\begin{array}{ccl}
 m_j(\lambda) + 1 & \colon j \in S, j+1 \notin S \\
 m_j(\lambda) - 1 & \colon j \notin S, j+1 \in S \\
 m_j(\lambda) & \colon j \in S, j+1 \in S \\
 m_j(\lambda) & \colon j \notin S, j+1 \notin S
\end{array}\right..$$
For example, for $k = 4$, $\lambda = 44211$ and $S = \set{1,3}$, we have $\lambda^S = 443111$. Note that $\lambda^S$ is not necessarily a partition, for example when $k = 4$, $\lambda = 44211$ and $S = \set{1,4}$, we have $m_3(\lambda^S) = m_3(\lambda) -1 < 0$. See the drawings on the left in Figure \ref{fig2}. We can also extend this definition to when $S$ is multiset with $\sigma_j$ copies of $j$: then let $\lambda^S$ denote the result of adding $\sigma_j$ cells in column $j$ to $\lambda$; in other words
$$m_j(\lambda^S) = m_j(\lambda) + \sigma_j - \sigma_{j+1}.$$
For example, for $\lambda = 44211$ and $S = \set{2^2,3}$, $\lambda^S = 44322$, but when $\lambda = 44211$ and $S = \set{2^2,4}$, $\lambda^S$ is not well defined since $m_3(\lambda^S) = m_3(\lambda) + \sigma_3 - \sigma_4 = 0 + 0 - 1 < 0$. We also extend this definition to a \emph{generalized multiset} $S$, where we allow $\sigma_j < 0$; this corresponds to the case of adding cells in some columns and removing cells in others. For example, for $\lambda = 44211$ and $S = \set{2^2,4^{-1}}$, we have $\lambda^S = 43222$. See the drawings on the right in Figure \ref{fig2}. \label{notation}  
\end{notation}

\begin{figure}[!ht]
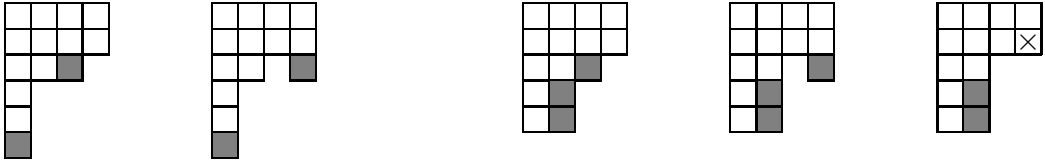
 
\begin{center}
 \ytableausetup{boxsize=0.8em}
\begin{ytableau}
{} & {} & {} & {} & \none & \none & \none & \none & & & & & \none & \none & \none & \none & \none & \none & \none & \none & {} & {} & {} & {} & \none & \none & \none & \none & & & &  & \none & \none & \none & \none & & & & \\
{} & {} & {} & {} & \none & \none & \none & \none & & & & & \none & \none & \none & \none & \none & \none & \none & \none & {} & {} & {} & {} & \none & \none & \none & \none & & & &  & \none & \none & \none & \none & & & & \times \\
{} & {} & *(gray) {} & \none & \none & \none & \none & \none & & & \none & *(gray)& \none & \none & \none & \none & \none & \none & \none & \none & {} & {} & *(gray) {} & \none & \none & \none & \none & \none & & & \none & *(gray) & \none & \none & \none & \none & & \\
{} & \none & \none & \none & \none & \none & \none & \none & & \none & \none & \none & \none & \none & \none & \none & \none & \none & \none & \none& {} & *(gray) & \none & \none & \none & \none & \none & \none & & *(gray) & \none & \none & \none & \none & \none & \none & & *(gray)\\
{} & \none & \none & \none & \none & \none & \none & \none & & \none & \none & \none & \none & \none & \none & \none & \none & \none & \none & \none&{} & *(gray) & \none & \none & \none & \none & \none & \none & & *(gray) & \none & \none & \none & \none & \none & \none & & *(gray)\\
*(gray) & \none & \none & \none & \none & \none & \none & \none & *(gray) & \none & \none & \none & \none & \none & \none & \none & \none & \none & \none & \none 
\end{ytableau}
\caption{Computing $\lambda^S$ for a set $S$. Computing $\lambda^S$ for a (generalized) multiset $S$.}\label{fig2}
\end{center}
\end{figure}

 When $\mu \subseteq \lambda$, it makes sense to define a multiset $S = \lambda - \mu$ by $\sigma_j = \lambda'_j - \mu'_j$; in particular, $\lambda^S - \lambda = S$ (when $\lambda^S$ is a partition). Even if $\mu \not\subseteq \lambda$, we can still define $S = \lambda - \mu$ by $\sigma_i = \lambda'_j - \mu'_j$ if we allow negative multiplicities in multisets.

\begin{remark}
The results mentioned above (Pieri rule, conjugate Pieri rule and Murnaghan-Nakayama rule) have something in common. For a fixed $n$, they give the expansion of the product of $s_\lambda$ with another symmetric function ($h_n$, $e_n$, or $p_n$). Furthermore, they are all stated so that they answer the following question: given a partition $\lambda$, for which $\nu \supseteq \lambda$ is the coefficient of $s_\nu$ in the product non-zero (and how to compute it)? It is trivial, but important for our purposes, to restate these results so that they answer the question: given a multiset $S$, for which $\lambda$ does $s_{\lambda^S}$ appear on the right-hand side (and with what coefficient)? The reader will easily check that for a multiset $S$ of size $n$, we have the following:
\begin{itemize}
 \item Pieri rule (horizontal strips): $s_{\lambda^S}$ appears in $s_\lambda h_n$ (equivalently: $\lambda^S/\lambda$ is a horizontal strip) if and only if $S$ is a set (\textit{i.e.}, $\sigma_j \leq 1$ for all $j$), and $m_j(\lambda) \geq \sigma_{j+1} - \sigma_j$ for $j \geq 1$;
 \item conjugate Pieri rule (vertical strips): $s_{\lambda^S}$ appears in $s_\lambda e_n$ (equivalently: $\lambda^S/\lambda$ is a vertical strip) if and only if and $m_j(\lambda) \geq \sigma_{j+1}$ for $j \geq 1$;
 \item Murnaghan-Nakayama rule (ribbons): $s_{\lambda^S}$ appears in $s_\lambda p_n$ (equivalently, $\lambda^S/\lambda$ is a ribbon) if and only if $\set{j \colon \sigma_j > 0}$ is an interval and $m_j(\lambda) = \sigma_{j+1} - 1$ if $\sigma_j > 0$, $\sigma_{j+1} > 0$, and $m_j(\lambda) > \sigma_{j+1} - 1$ if $\sigma_j = 0$, $\sigma_{j+1} > 0$; furthermore, when this is satisfied, the coefficient with which $s_{\lambda^S}$ appears in $s_\lambda p_n$ is independent of $\lambda$ and equals $(-1)^{\sum_j (\sigma_j-1)}$, where the sum is over $j$ with $\sigma_j > 0$.
\end{itemize} \label{remark}
\end{remark}

\subsection{$k$-Schur functions}

For $k$-bounded partitions $\lambda,\mu$, we say that $\lambda/\mu$ is a \emph{($k$-)weak horizontal strip} if $\lambda/\mu$ is a horizontal strip and $\lambda^{(k)}/\mu^{(k)}$ is a vertical strip. We say that \emph{$\lambda$ covers $\mu$ in the weak order} if $\lambda/\mu$ is a weak horizontal strip of size $1$. A \emph{($k$-)weak semistandard Young tableau} of shape $\lambda$ and weight $\mu$ is a sequence of partitions $\lambda^{0} \subseteq \lambda^{1} \subseteq \ldots \subseteq \lambda^{m}$ such that $\lambda^{0} = \emptyset$, $\lambda^{m} = \lambda$, and $\lambda^{i}/\lambda^{i-1}$ is a weak horizontal strip of size $\mu_i$. Define the \emph{($k$-)weak Kostka number} $K_{\lambda\mu}^{(k)}$ as the number of $k$-weak semistandard Young tableaux of shape $\lambda$ and weight $\mu$. Again, $K_{\lambda \lambda}^{(k)} = 1$ and $K_{\lambda \mu}^{(k)} = 0$ unless $\lambda \geq \mu$. In other words, the matrix $(K_{\lambda\mu}^{(k)})_{\lambda,\mu \in \Par(n,k)}$ is upper-triangular with $1$'s on the diagonal (in any linear extension of the dominance order) and hence invertible. Therefore we can define \emph{$k$-Schur functions} by

$$h_\mu = \sum_\lambda K_{\lambda\mu}^{(k)} s_\lambda^{(k)}.$$

Denote by $\Lambda^n_k$ the vector space spanned by $\set{h_\lambda \colon \lambda \in \Par(n,k)}$, and the algebra $\Lambda^0_k \oplus \Lambda^1_k \oplus \cdots$ by $\Lambda_k$. By construction, $\set{s_\lambda^{(k)} \colon \lambda \in \Par(n,k)}$ is a basis of $\Lambda^n_k$.

\medskip

By definition, $k$-Schur functions satisfy the \emph{Pieri rule}: for a $k$-bounded partition $\lambda$ and $n \leq k$, we have
\begin{equation} \label{kpieri}
s_\lambda^{(k)} h_n = \sum_\nu s_\nu^{(k)},
\end{equation}

where the sum is over $k$-bounded partitions $\nu$ for which $\nu/\lambda$ is a weak horizontal strip of size $n$.

\medskip

For $k$-bounded partitions $\lambda,\mu$, we say that $\lambda/\mu$ is a \emph{($k$-)weak vertical strip} if $\lambda/\mu$ is a vertical strip and $\lambda^{(k)}/\mu^{(k)}$ is a horizontal strip. In \cite[Theorem 33]{LM07}, the \emph{conjugate Pieri rule} is proved: for a $k$-bounded partition $\lambda$ and $n \leq k$, we have
\begin{equation} \label{kconjpieri}
s_\lambda^{(k)} e_n = \sum_\nu s_\nu^{(k)},
\end{equation}
where the sum is over $k$-bounded partitions $\nu$ for which $\nu/\lambda$ is a weak vertical strip of size $n$.

\medskip

If $k$ is large enough, then $k$-conjugates are the same as conjugates, $k$-weak horizontal strips are just horizontal strips, and $k$-weak semistandard Young tableaux are just semistandard Young tableaux. Therefore $k$-Schur functions converge to the usual Schur functions as $k$ increases. More specifically, if $\lambda$ is a $k$-bounded partition that is also a $(k+1)$-core, then $s_\lambda^{(k)} = s_\lambda$. The theory of $k$-Schur functions has been the focus of much research in the last decade. The properties of $k$-Schur functions are usually analogous to (but more complicated than) the properties of Schur functions, but they exhibit an interesting multiplicativity property that is absent in the theory of Schur functions. Namely, we have
\begin{equation} \label{eqn}
s_\lambda^{(k)} s_{l^{k+1-l}}^{(k)} = s^{(k)}_{\lambda^{(k)} \cup l^{k+1-l}}
\end{equation}
for all $l$, $1 \leq l \leq k$, a statement we give a new proof of in Section \ref{application}. Note that the partition $l^{k+1-l}$ is a $(k+1)$-core and that therefore $s_{l^{k+1-l}}^{(k)} = s_{l^{k+1-l}}$. This property enables us to write any $k$-Schur functions in terms of $k$-Schur functions corresponding to $k$-irreducible partitions. For example, we have $s_{44333222211111111111}^{(4)} = s_4^2 s_{33} s_{222} s_{1111}^2 s_{32111}^{(4)}$. 

\medskip

Finally, let us mention that $k$-Schur functions in full generality possess a parameter $t$, and our $k$-Schur functions are the result of specializing $t \to 1$ (compare with Hall-Littlewood polynomials). See \cite[\S 2]{llmssz} for more information about $k$-Schur functions.

\section{Residue and quotient tables and the meta-conjecture} \label{residue}

For a $k$-bounded partition $\lambda$, the \emph{residue table of $\lambda$} is the upper-triangular $k \times k$ matrix $R = R(\lambda) = (r_{ij})_{1 \leq i \leq j \leq k}$ defined as follows:
\begin{itemize}
 \item $r_{ii} = m_i(\lambda) \md (k+1-i)$
 \item $r_{ij} = (m_j(\lambda) + r_{i,j-1}) \md (k+1-j)$
\end{itemize}
The \emph{quotient table of $\lambda$} is the upper-triangular $k \times k$ matrix $Q = Q(\lambda) = (q_{ij})_{1 \leq i \leq j \leq k}$ defined as follows:
\begin{itemize}
 \item $q_{ii} = m_i(\lambda) \dv (k+1-i)$
 \item $q_{ij} = (m_j(\lambda) + r_{i,j-1}) \dv (k+1-j)$
\end{itemize}

\medskip

It is obvious from the construction of $R$ and $Q$ that 
\begin{align*}
 (k+1-i)q_{ii}+r_{ii} &= m_i(\lambda), \\
 (k+1-j)q_{ij} + r_{ij} &= m_j(\lambda) + r_{i,j-1}.
\end{align*}
Note that it does \emph{not} hold that $r_{ij} = (m_i(\lambda) + \cdots + m_j(\lambda)) \md (k+1-j)$ in general. 

\begin{example}
 For $k = 4$, $\lambda = 44432211111$, the residue and quotient tables are 
$$\begin{matrix}
 1 & 0 & 1 & 0 \\
   & 2 & 1 & 0 \\
   &   & 1 & 0 \\
   &   &   & 0
\end{matrix} \qquad  \mbox{and} \qquad\begin{matrix}
  1 & 1 & 0 & 4 \\
   & 0 & 1 & 4 \\
   &   & 0 & 4 \\
   &   &   & 3
\end{matrix}$$
 respectively.
\end{example}

These tables were introduced in \cite{fk} (with the role of rows and columns reversed) to describe strong covers (see \cite[Theorem 5.2]{fk} and Theorem \ref{strong}) and seem to play a very important role in the theory of $k$-bounded partitions and $k$-Schur functions.

\medskip

Indeed, the aim of this paper is to provide (further) supporting evidence for the following (admittedly vague) statement.

\begin{metaconjecture}
 (Almost) everything in the theory of $k$-bounded partitions and $k$-Schur functions can be expressed in an elegant way in terms of residue tables.
\end{metaconjecture}

In particular, we prove or conjecture the following.
\begin{itemize}
 \item In the following section, we show how to describe the $k$-conjugate of a partition via quotient tables and the size of the corresponding $(k+1)$-core in terms of residue and quotient tables.
 \item In Section \ref{strips}, we describe strong and weak covers, as well as weak horizontal and vertical strips, in terms of residue and quotient tables, and use this to restate the Pieri rule for $k$-Schur functions.
 \item In Section \ref{application}, we use the description of weak horizontal strips to reprove equation \eqref{eqn} in a simpler way.
 \item It seems that the (very complicated) definition of a \emph{$k$-ribbon} from \cite{bsz} has a better description in terms of residue tables, see Section \ref{mn}.
 \item Multiplication of $s_\lambda^{(k)}$ with $s_{l^{k-l},l-1}$  yields a sum of the form $\sum_\nu s_{\nu}^{(k)}$, with a simple condition on the residue table of $\lambda$ determining which $\nu$'s appear in the sum; this can be used to reprove the Monk's formula for quantum Schubert polynomials. See Subsection \ref{almost}.
 \item The concept of \emph{splitting} of a $k$-bounded partition has a description with residue tables, and it seems plausible that this could be used to give a more elementary proof of a theorem due to Denton (\cite[Theorem 1.1]{denton}), see Subsection \ref{split}.
 \item At least one special case of LLMS insertion for standard tableaux (see \cite[\S 10.4]{LLMS10}) can be described in terms of residue tables, see Subsection \ref{llms}.
 \item Finally, there is ample evidence that one of the major unsolved problems in the theory of $k$-Schur functions, a description of $k$-Littlewood-Richardson coefficients, is possible via residue tables; see Section \ref{kLR}.
\end{itemize}

The author was unable to (conjecturally) describe the expansion of $k$-Schur functions in terms of Schur functions, or the expansion of a $k$-Schur function in terms of $(k+1)$-Schur functions (which we call \emph{$k$-branching}), via the residue and quotient tables, and these could be putting the ``almost'' into the Meta-conjecture. Note that there is a conjectural expansion of $k$-Schur functions in terms of Schur functions using atoms, see \cite{llm03} (alternatively, if one takes the definition of $k$-Schur functions via atoms, then the definition above is a conjecture), and $k$-branching in solved by \cite[Theorem 2]{kbranching}.

\section{Computing the $k$-conjugate and the size of the $(k+1)$-core} \label{s:kconj}

The reader might be wondering whether the entries in the residue and quotient tables have a specific meaning. The following should answer that question.

\begin{proposition} \label{meaning}
 For a $k$-bounded partition $\lambda$ and $j$, $1 \leq j \leq k$, denote the partition $\langle 1^{m_1(\lambda)},\ldots,j^{m_j(\lambda)} \rangle$ by $\lambda_{(j)}$, and denote the residue and quotient tables of $\lambda$ by $R$ and $Q$, respectively. Then for $i \leq j$, the number of parts of $\lambda_{(j)}^{(k)}$ equal to $k + 1 - i$ is the sum of the entries in column $i$ of $Q$, and the parts of $\lambda_{(j)}^{(k)}$ that are at most $k - j$ are precisely the non-zero entries in column $j$ of $R$.
\end{proposition}
\begin{proof}

  We prove the statement by induction on $j$. There is only one entry in column $1$ of $R$, $m_1(\lambda) \md k$, and only one entry in column $1$ of $Q$, $m_1(\lambda) \dv k$. On the other hand, $\lambda_{(1)} = \langle 1^{m_1(\lambda)}\rangle$, and the $(k+1)$-core corresponding to this $k$-bounded partition has $k$ copies of part $i$ for $i = 1,\ldots,m_1(\lambda) \dv k$, and $m_1(\lambda) \md k$ copies of part $(m_1(\lambda) \dv k) + 1$. For example, for $k = 3$ and $m_1(\lambda) = 8$, we get
 $\: \:\ytableausetup{boxsize=0.3em}
\begin{ytableau}
  \none & \none &  \\
\none & \none & \\
 \none & \\
 \none & \\ 
 \none & \\
 {} \\
  {}\\
  {}
 \end{ytableau}$.\\

That means that $\lambda_{(1)}^{(k)}$ contains $m_1(\lambda) \dv k$ copies of $k$, and one copy of $m_1(\lambda) \md k$. This proves the statement for $j = 1$.\\
 Assume that the statement holds for $j-1$; \textit{i.e.}, the non-zero entries of column $j-1$ of $R$ are the parts of $\lambda_{(j-1)}^{(k)}$ that are at most $k + 1 - j$, and for $i \leq j-1$, the number of parts of $\lambda_{(j-1)}^{(k)}$ equal to $k + 1 - i$ is the sum of the entries in column $i$ of $Q$. By the construction of $\mf c^{(k)}$, the parts of $\lambda_{(j-1)}^{(k)}$ that are at most $k + 1 - j$ are top-justified; indeed, if a cell of the diagram of $\mf c^{(k)}(\lambda_{(j-1)})$ with hook-length $\geq k+2$ lies immediately above $i'$ cells with hook-length $\leq k$ and immediately to the left of $i''$ cells with hook-length $\leq k$, then $i' + i'' \geq k + 1$, so $i'' \leq j-1$ implies that $i' \geq k + 2 - j$.\\
 Now prove the statement by induction on $m = m_j(\lambda) = m_j(\lambda_{(j)})$. If $\lambda$ contains no parts equal to $j$, then $\lambda_{(j-1)}^{(k)} = \lambda_{(j)}^{(k)}$. The parts of $\lambda_{(j)}^{(k)}$ that are at most $k - j$ are precisely the parts of $\lambda_{(j-1)}^{(k)}$ that are at most $k + 1 - j$ and that are not equal to $k + 1 - j$, by induction, these are the elements $< k + 1 - j$ in column $j-1$ of $R$. On the other hand, the elements of column $j$ of $R$ are $r_{jj} = m_j(\lambda) \md (k + 1 - j) = 0$ and $r_{ij}=r_{i,j-1} \md (k + 1 - j)$, which is $r_{i,j-1}$ if $r_{i,j-1} < k + 1 - j$ and $0$ if $r_{i,j-1} = 0$. Therefore the non-zero elements of column $j$ are precisely the non-zero elements of column $j-1$ that are not equal to $k + 1 - j$. Furthermore, the number of parts of $\lambda_{(j)}^{(k)}$ equal to $k + 1 - i$ for $i < j$ is by induction equal to the sum of the entries in column $i$ of $Q$, and the number of parts of $\lambda_{(j)}^{(k)}$ equal to $k + 1 - j$ is, again by induction, the number of entries of column $j-1$ of $R$ equal to $k+1-j$. But since $q_{jj} = m_j(\lambda) \dv (k + 1 - j) = 0$ and $q_{ij}=r_{i,j-1} \md (k + 1 - j)$, $q_{ij} = 0$ if $r_{i,j-1} < k+1-j$ and $q_{ij} = 1$ if $r_{i,j-1} = k+1-j$. Therefore the number of parts of $\lambda_{(j)}^{(k)}$ equal to $k + 1 - j$ is the sum of the column $j$ of $Q$. This concludes the proof of the base of (inner) induction.\\
 Now assume that $m > 0$ and that the statement holds for $\mu$ which has $m_j(\mu) = m - 1$, $m_i(\mu) = m_i(\lambda)$ for $i < j$, $m_i(\mu) = 0$ for $i > j$, and residue and quotient tables $R'$ and $Q'$. The first $j-1$ columns of $R'$ and $Q'$ are the same as of $R$ and $Q$. If $r'_{ij} < k-j$, then $r_{ij} = r'_{ij}+1$ and $q_{ij} = q'_{ij}$, and if $r'_{ij} = k - j$, then $r_{ij} = 0$ and $q_{ij} = q'_{ij}+1$. But when we add a new row of length $j$ to the core $\mf c^{(k)}(\mu)$, the columns of length $\geq k+1-j$ are unchanged, the columns of length $k-j$ are changed to columns of length $k + 1 -j$, and the columns of length $< k - j$ get one extra cell. This means that the columns of length $\leq k-j$ are the non-zero entries of column $j$ of $R$, and the number of columns of length $k+1-j$ is the sum of column $j$ of $Q$.
\end{proof}

\begin{example}
 Take $k = 4$ and $\lambda = 44432211111$. Then $\lambda_{(1)} = 11111$, $\lambda_{(2)} = 2211111$, $\lambda_{(3)} = 32211111$, $\lambda_{(4)} = 44432211111$, and their $k$-conjugates are $41$, $432$, $432111$, $432111111111111111$. The parts of $\lambda_{(j)}^{(4)}$ that are at most $4-j$ are $1$, $2$, $111$, none for $j = 1,2,3,4$, and these are precisely the non-zero entries of the columns of $\begin{smallmatrix}
 1 & 0 & 1 & 0 \\
   & 2 & 1 & 0 \\
   &   & 1 & 0 \\
   &   &   & 0
\end{smallmatrix}$. Furthermore, the sums of columns of $\begin{smallmatrix}
  1 & 1 & 0 & 4 \\
   & 0 & 1 & 4 \\
   &   & 0 & 4 \\
   &   &   & 3
\end{smallmatrix}$ are $1$, $1$, $1$, $15$, and indeed $\lambda_{(j)}^{(4)}$ has $1$ part equal to $4$ for $j=1,2,3,4$, $\lambda_{(j)}^{(4)}$ has $1$ part equal to $3$ for $j=2,3,4$, $\lambda_{(j)}^{(4)}$ has $1$ part equal to $2$ for $j=3,4$, and $\lambda_{(j)}^{(4)}$ has $15$ parts equal to $1$ for $j=4$.
\end{example}

\begin{remark}
 The proposition tells us that the residue and quotient tables essentially describe the bijection $\mf c^{(k)}$, the process of turning the given $k$-bounded partition into a $(k+1)$-core. The first column of the residue and quotient tables describe what happens after we add the $1$'s, the second column of the residue table and the first two columns of the quotient table describe what happens when we add the $1$'s and the $2$'s, etc. \\
 However, the residue and quotient tables give \emph{more} information than that. Namely, instead of just giving the sizes of the columns of $\lambda_{(j)}^{(k)}$, they also give a special ordering of the sizes. Indeed, one could say that cores are ``incomplete'' descriptions of $k$-bounded partitions since the order of the entries of the columns is lost. In all descriptions of results on $k$-bounded partitions and $k$-Schur functions to follow, the exact position in a column of the residue or quotient table plays a crucial role.
\end{remark}

\medskip

Since $\lambda_{(k)} = \lambda$, we have the following result.

\begin{corollary} \label{kconj}
 For a $k$-bounded partition $\lambda$ with quotient table $Q$, we have 
 $$m_j(\lambda^{(k)}) = \sum_{i=1}^{k+1-j} q_{i,k+1-j}. \eqno \qed$$
\end{corollary}

\begin{example}
 Take $k = 4$ and $\lambda = 44432211111$ as in the previous example. Then the sums of columns of $\begin{smallmatrix}
  1 & 1 & 0 & 4 \\
   & 0 & 1 & 4 \\
   &   & 0 & 4 \\
   &   &   & 3
\end{smallmatrix}$ are $1$, $1$, $1$, $15$, \textit{i.e.}\ $44432211111^{(4)} = 432111111111111111$.
\end{example}

When we add a new row of length $j$ in the construction of $\mf c^{(k)}(\lambda)$, the number of new cells with hook-length $> k$ is equal to the number of columns of length $>k-j$, and we know that these are enumerated by the sum of the first $j$ columns of the quotient table. Therefore it should not come as a surprise that the residue and quotient tables can also be used to compute the number of cells of $\mf c^{(k)}(\lambda)$ with hook-length $>k$. Let us remark that since $|\lambda| = \sum_{j=1}^k j m_j(\lambda) = \sum_{j=1}^k j(r_{jj} + (k+1-j) q_{jj})$, we could instead give a formula for size of the $(k+1)$-core corresponding to a $k$-bounded partition.

\begin{theorem} \label{size}
 For a $k$-bounded partition $\lambda$ with residue table $R$ and quotient table $Q$, we have
 $$|\mf c^{(k)}(\lambda)| - |\lambda| = \!\!\!\! \sum_{1 \leq i \leq j \leq k}\!\!\!\! r_{ij} q_{ij}+ \!\!\!\!\sum_{1 \leq i \leq j < h \leq k}\!\!\!\! r_{hh}q_{ij} + \!\!\!\!\sum_{1 \leq i \leq j \leq h \leq k}\!\!\!\! (k+1-h)q_{hh}q_{ij} - \!\!\sum_{1 \leq i \leq k}\!\! i (k+1-i) \left(\begin{smallmatrix}q_{ii}+1 \\ 2\end{smallmatrix}\right) .$$
\end{theorem}

\begin{example}
 For $k = 4$ and $\lambda = 44432211111$, the theorem yields $|\mf c^{(4)}(\lambda)| - |\lambda| = [1 \cdot 1 + 1 \cdot 1]+[2 \cdot 1 + 1 \cdot 2] + [4 \cdot 1 \cdot 1+1 \cdot 3 \cdot 18] - [1 \cdot 4 \cdot \binom 22 + 4 \cdot 1 \cdot \binom 4 2] = 36$, as confirmed by the computation $\mf c^{(4)}(\lambda) = (18)(14)(10)63321111$, $|\mf c^{(4)}(\lambda)| - |\lambda| = 60 - 24 = 36$.\\
 As a more general example, the theorem is saying that for a $3$-bounded partition $\lambda$ with residue table $R$ and quotient table $Q$,
 \begin{align*}
 |\mf c^{(k)}(\lambda)| \! - \!|\lambda| &= \! \left[ r_{11}q_{11} + r_{12}q_{12} + r_{13}q_{13} + r_{22}q_{22} + r_{23}q_{23} + r_{33}q_{33} \right] \! + \! [r_{22}q_{11} + r_{33}(q_{11}+q_{12}+q_{22})]\\
&+ [3q_{11}q_{11} + 2q_{22}(q_{11}+q_{12}+q_{22}) + q_{33}(q_{11}+q_{12}+q_{13}+q_{22}+q_{23}+q_{33})] \\
&- \left[ 3 \left(\begin{smallmatrix}q_{11}+1 \\ 2\end{smallmatrix}\right)+ 4 \left(\begin{smallmatrix}q_{22}+1 \\ 2\end{smallmatrix}\right)+ 3 \left(\begin{smallmatrix}q_{33}+1 \\ 2\end{smallmatrix}\right) \right].\end{align*}
 It is also clear from the theorem that for a $k$-irreducible partition $\lambda$, with $q_{ii} = 0$ for all $i$, we have
 $$|\mf c^{(k)}(\lambda)| - |\lambda| = \!\!\!\! \sum_{1 \leq i < j \leq k}\!\!\!\! r_{ij} q_{ij}+ \!\!\!\!\sum_{1 \leq i < j < h \leq k}\!\!\!\! r_{hh}q_{ij}  .$$
\end{example}

The theorem is proved in Section \ref{proofs}.

\section{Strong and weak covers, weak horizontal strips, and weak vertical strips} \label{strips}

As mentioned in Section \ref{residue}, residue and quotient tables were introduced in \cite{fk} to describe strong covers. We restate the description since the definitions of residue and quotient tables used here are a bit different, and because the $\lambda^S$-notation makes the description slightly more elegant.

\begin{definition}
 For $k$-bounded partitions $\lambda$ and $\mu$ satisfying $|\lambda| = |\mu| + 1$, we say that \emph{$\lambda$ covers $\mu$ in the strong order with multiplicity $d$} if ${\mf c^{(k)}}(\mu) \subseteq {\mf c^{(k)}}(\lambda)$, and ${\mf c^{(k)}}(\lambda)/{\mf c^{(k)}}(\mu)$ has $d \geq 1$ connected components (which are necessarily ribbons and translates of each other).
\end{definition}

\begin{theorem}{\cite[Theorem 5.2]{fk}} \label{strong}
 For a generalized multiset $S$ of size $-1$, $\lambda$ covers $\lambda^S$ in the strong order if and only if, for some $1 \leq I < J \leq k+1$:
 \begin{itemize}
  \item $\sigma_i = 0$ if $i \neq I,J$,
  \item $r_{I,j} > \sigma_J$ for $j = I,\ldots,J-2$,
  \item $r_{I,J-1} = \sigma_J$,
 \end{itemize}
 where $R = (r_{ij})_{1 \leq i \leq j \leq k}$ is the residue table of $\lambda$. Furthermore, the multiplicity of this cover relation is $q_{I,I} + \cdots + q_{I,J-1}$, where $Q = (q_{ij})_{1 \leq i \leq j \leq k}$ is the quotient table of $\lambda$ (if this sum is $0$, $\lambda$ does not cover $\lambda^S$ in the strong order). \qed
\end{theorem}

In other words, to find all partitions that $\lambda$ covers, find entries $r_{I,J-1}$ in the residue table that are strictly smaller than the entries to its left, and for every such $I,J$, add $r_{I,J-1}$ cells in column $J$ of $\lambda$, and remove $r_{I,J-1}+1$ cells in column $I$ of $\lambda$; the corresponding multiplicity is the sum of the entries in row $I$ up to column $J-1$ of the quotient table. Since the $k$-column of the residue table contains only zeros, we never have to add cells in column $k+1$, even when $J = k + 1$. 

\begin{example} 
 For $k = 4$ and $\lambda = 44432211111$ from the previous example, there are eight entries in the residue table of $\lambda$ that are strictly smaller than all the entries to its left: $r_{11} = 1,r_{12} = 0,r_{22}=2,r_{23}=1,r_{24}=0,r_{33}=1,r_{34}=0,r_{44}=0$. Therefore $\lambda$ covers:
 \begin{itemize}
  \item $\lambda^{\set{1^{-2},2^1}} = 444322211$ with multiplicity $1$;
  \item $\lambda^{\set{1^{-1},3^0}} = 4443221111$ with multiplicity $1 + 1 = 2$;
  \item $\lambda^{\set{2^{-3},3^{2}}}$ (not a valid partition) with multiplicity $0$;
  \item $\lambda^{\set{2^{-2},4^{1}}} = 44441111111$ with multiplicity $0 + 1 = 1$;
  \item $\lambda^{\set{2^{-1},5^{0}}} = 44432111111$ with multiplicity $0 + 1 + 4 = 5$;
  \item $\lambda^{\set{3^{-2},4^{1}}}$ (not a valid partition) with multiplicity $0$;
  \item $\lambda^{\set{3^{-1},5^{0}}} = 44422211111$ with multiplicity $0 + 4 = 4$;
  \item $\lambda^{\set{4^{-1},5^{0}}} = 44332211111$ with multiplicity $3$;
 \end{itemize}
\end{example}

Residue tables also enable us to give a truly concise description of weak horizontal strips (and therefore of weak covers, which are weak horizontal strips of size $1$). 

\begin{theorem} \label{whs}
 For a $k$-bounded partition $\lambda$ and $S \subseteq [k] = \set{1,\ldots,k}$, $\lambda^S/\lambda$ is a weak horizontal strip if and only if $r_{i,j-1} > 0$ for $i \notin S$, $j \in S$, $i < j$, where $R = (r_{ij})_{1 \leq i \leq j \leq k}$ is the residue table of $\lambda$. In particular, $\lambda^{\set j}$ covers $\lambda$ in the weak order if and only if $r_{1,j-1},\ldots,r_{j-1,j-1} > 0$.
\end{theorem}

\begin{example}
 The residue table of the $4$-bounded partition $\lambda = 44211$ is $\begin{smallmatrix}
 2 & 0 & 0 & 0 \\
   & 1 & 1 & 0 \\
   &   & 0 & 0 \\
   &   &   & 0
\end{smallmatrix}$. Since $r_{22} > 0$, $\lambda^{\set{1,3}}/\lambda = 443111/44211$ is a weak horizontal strip; indeed, the $4$-conjugates of $\lambda^{\set{1,3}}$ and $\lambda$ are $311111111111$ and $3111111111$. On the other hand, $r_{33} = 0$, so $\lambda^{\set{1,4}}/\lambda$ is not a weak horizontal strip (indeed, $\lambda^{\set{1,4}}$ is not even a partition). Since $r_{12} = 0$, $\lambda^{\set{2,3}}/\lambda= 44321/44211$ is also not a weak horizontal strip, even though it is a horizontal strip; indeed, the $4$-conjugate of $\lambda^{\set{2,3}}$ is $221111111111$, and $221111111111/3111111111$ is obviously not a vertical strip.\\
For $j = 1$ and $j = 2$, all the entries of column $j-1$ are non-zero (for $j = 1$, this is true vacuously and for all $k$-bounded partitions). Therefore $\lambda$ is covered by two elements in the weak order, $\lambda^{\set 1} = 442111$ and $\lambda^{\set 2} = 44221$.
\end{example}

For the proof of Theorem \ref{whs}, we need the following lemma.

\begin{lemma} \label{lemma}
 Denote the residue and quotient tables of $\lambda$ (resp., $\lambda^S$) by $R$ and $Q$ (resp., $R'$ and $Q'$), and write $m_j$ (resp., $m'_j$) for $m_j(\lambda)$ (resp., $m_j(\lambda^S)$). 
 \begin{enumerate} 
 \renewcommand{\labelenumi}{(\alph{enumi})}
  \item Suppose that $r_{i,j-1} > 0$ for all $i \notin S$, $j \in S$, $i < j$. For $i \in S$, define 
$$h(i) = \min \{ h \geq i \colon h+1 \notin S, r_{ih} = k - h \}$$ 
(since $h = k$ satisfies the condition, $h(i)$ is well defined).
 Then
 \begin{align*}
 r'_{ij} &= \left\{ 
\begin{array}{ccl}
 r_{ij}-1 & \colon & i \notin S, j + 1 \in S \\
 r_{ij} & \colon & i \notin S, j + 1 \notin S \\
 r_{ij} & \colon & i \in S, j + 1 \in S, j < h(i) \\
 r_{ij}-1 & \colon & i \in S, j + 1 \in S, j > h(i) \\
 r_{ij}+1 & \colon & i \in S, j + 1 \notin S, j < h(i) \\
 0 & \colon & i \in S, j + 1 \notin S, j = h(i) \\
 r_{ij} & \colon & i \in S, j + 1 \notin S, j > h(i)
\end{array}
\right. \\
 q'_{ij} &= \left\{ 
\begin{array}{ccl} 
q_{ij}+1 & \colon & i \in S, j = h(i) \\ 
q_{ij}-1 & \colon & i \in S, j = h(i)+1 \\ 
q_{ij} & \colon & \mbox{otherwise} 
\end{array}\right..
\end{align*}
Furthermore, if $i \in S$, $j > h(i)$, then $r_{ij} = r_{h(i)+1,j}$ (and in particular, $r_{i,j-1} > 0$ if $j \in S$) and $r'_{ij} = r'_{h(i)+1,j}$.
 \item Suppose that $r_{i,j-1} = 0$ for some $i \notin S$, $j \in S$, $i < j$. Then $\lambda^{(k)} \not\subseteq (\lambda^S)^{(k)}$.
\end{enumerate}
\end{lemma}

The proof of the lemma is given in Section \ref{proofs}.

\begin{proof}[Proof of Theorem \ref{whs}]
 The ``only if'' part of the theorem follows immediately from part (b) of Lemma \ref{lemma}, let us prove the ``if'' part.
 Clearly $\sigma_j \leq 1$ for all $j$, and if $j+1 \in S, j \notin S$, then $m_j(\lambda) \geq r_{jj} > 0$, so $\lambda^{S}/\lambda$ is a horizontal strip (see Remark \ref{remark}). It remains to prove that $(\lambda^{S})^{(k)}/\lambda^{(k)}$ is a vertical strip. By Lemma \ref{lemma} and Corollary \ref{kconj}, we have $(\lambda^S)^{(k)} = \left(\lambda^{(k)}\right)^T$, where $T$ is a multiset with $\tau_{j} = |\set{i \colon i \in S, h(i) = k + 1 - j}|$ copies of $j$. By Remark \ref{remark}, we need to prove that $\tau_j \leq m_{j-1}(\lambda^{(k)})$. But whenever $h(i) = k+1-j$ for $i \in S$, we have $q'_{i,k+2-j} = q_{i,k+2-j}-1$, and in particular, $q_{i,k+2-j} \geq 1$. By Corollary \ref{kconj}, $m_{j-1}(\lambda^{(k)}) \geq |\set{i \colon i \in S, h(i) = k + 1 - j}| = \tau_j$.
\end{proof}

\begin{corollary} \label{pieri}
 For a $k$-bounded partition $\lambda$ and $1 \leq n \leq k$, we have
 $$s_\lambda^{(k)} h_n = \sum_S s_{\lambda^S}^{(k)},$$
 where the sum is over all sets $S \subseteq [k]$ of size $n$ with the property $r_{ij} > 0$ for $i \notin S$, $j+1 \in S$, $i \leq j$. In particular, $s_\lambda^{(k)} h_k = s_{\lambda \cup k}^{(k)}$.
\end{corollary}
\begin{proof}
 The first statement in just a restatement of \eqref{kpieri}, the Pieri rule for $k$-Schur functions, using Theorem \ref{whs}. The only $k$-subset of $[k]$ is $[k]$ itself, and $r_{ij} > 0$ for $i \notin S$, $j+1 \in S$, $i \leq j$, is clearly satisfied. It is also obvious that $\lambda^{[k]} = \lambda \cup k$.
\end{proof}

The description of weak vertical strips is more complicated. The following theorem tells us how to use the residue table to determine whether or not $\lambda^S/\lambda$ is a weak vertical strip for a multiset $S$ with $\sigma_j$ copies of $j$. 

\begin{theorem} \label{wvs}
 For a $k$-bounded partition $\lambda$ and a multiset $S \subseteq [k]$ of size $\leq k$ with $\sigma_i$ copies of $i$, $\lambda^S/\lambda$ is a weak vertical strip if and only if for $i < j$, $\sigma_j > 0$, we have
 $$\sigma_j \leq r_{i,j-1} \leq k + 1 - j - \sigma_i + |\{ h \colon i \leq h \leq j-2, r_{ih} > k-h-\sigma_i, r_{ih} > r_{i,j-1} \} |. $$
\end{theorem}

In particular, if $\sigma_j > 0$, then $\sigma_j \leq r_{j-1,j-1} \leq k - j + 1 - \sigma_{j-1}$. Also, if $S$ is a set, then $\lambda^S/\lambda$ is a weak vertical strip if and only if:
\begin{itemize}
 \item for $i < j$, $j \in S$, we have $1 \leq r_{i,j-1}$;
 \item for $i < j$, $i,j \in S$, we have $r_{i,j-1} = k + 1 -j \Rightarrow r_{ih} = k-h$ for some $h$, $i \leq h \leq j-2$.
\end{itemize}
In other words, $r_{i,j-1}$ is allowed to be maximal for $i,j \in S$, but that has to be ``compensated for'' by another $r_{ih}$ to the left of $r_{i,j-1}$ being maximal as well.

\begin{example}
 Take $k = 4$, $S = \set{1,3}$, $\lambda = 22111$ and $\mu = 22$. The residue tables of $\lambda$ and $\mu$ are $\begin{smallmatrix}
    3 & 2 & 0 & 0 \\
   & 2 & 0 & 0 \\
   &   & 0 & 0 \\
   &   &   & 0
   \end{smallmatrix}$ and $ \begin{smallmatrix}
   0 & 2 & 0 & 0 \\
   & 2 & 0 & 0 \\
   &   & 0 & 0 \\
   &   &   & 0
 \end{smallmatrix}$. While $r_{12} \geq 1$, $r_{22} \geq 1$ and $r_{12} = 2$ for both $\lambda$ and $\mu$, we have $r_{11} = 3$ for $\lambda$ and $r_{11} < 3$ for $\mu$. Therefore $\lambda^S/\lambda=321111/22111$ is a weak vertical strip while $\mu^S/\mu = 321/22$ is not.
\end{example}

The proof of Theorem \ref{wvs} is similar to the proof of Theorem \ref{whs} and is omitted.

\section{An application: multiplication with a $k$-rectangle} \label{application}

A $k$-rectangle is the Schur (and $k$-Schur) function $s_{l^{k+1-l}} = s_{l^{k+1-l}}^{(k)}$. Multiplication with a $k$-rectangle is very special; the following theorem is known (see \cite[Theorem 40]{LM07}), but we give a new and more elementary proof.

\begin{theorem}
 For a $k$-bounded partition $\lambda$ and $l$, $1 \leq l \leq k$, we have
 $$s_\lambda^{(k)} s_{l^{k+1-l}} = s_{\lambda \cup l^{k+1-l}}^{(k)}.$$
\end{theorem}
\begin{proof}
 We prove the statement by induction on $n = |\lambda|$. For $\lambda = \emptyset$, this is the statement that $s_{l^{k+1-l}} = s_{l^{k+1-l}}^{(k)}$, which follows from the fact that $l^{k+1-l}$ is a $(k+1)$-core.\\
 Assume that we have proved the statement for all $k$-bounded partitions of size $<n$; we prove the statement for partitions of size $n$ by reverse induction on $\lambda_1$.\\
 If $n < k$, then the maximal possible $\lambda_1$ is $n$, when $\lambda = n$ and $s_\lambda^{(k)} = s_n = h_n$. By Corollary \ref{pieri}, $s_\lambda^{(k)} s_{l^{k+1-l}} = s_{l^{k+1-l}}^{(k)} h_n = \sum_{S} s_{(l^{k+1-l})^S}^{(k)}$, where the sum is over all subsets $S$ of $[k]$ of size $n$ for which $r_{i,j-1} > 0$ for all $i < j$, $i \notin S$, $j \in S$ for $R$ the residue table of $l^{k+1-l}$. But it is clear that the residue table of $l^{k+1-l}$ contains only zeros, so the condition is satisfied if and only if $S = [n]$. Note that $(l^{k+1-l})^{S} = l^{k+1-l} \cup n$ in this case. Therefore $s_\lambda^{(k)} s_{l^{k+1-l}} = s_{\lambda \cup l^{k+1-l}}^{(k)}$.\\
 If $n \geq k$, the maximal possible $\lambda_1$ is $k$. Write $\lambda'$ for $(\lambda_2,\ldots,\lambda_\ell)$. By Corollary \ref{pieri}, $s_\lambda^{(k)} = s_{\lambda'}^{(k)} h_k$, so
 $$s_\lambda^{(k)} s_{l^{k+1-l}} = (s_{\lambda'}^{(k)} h_k) s_{l^{k+1-l}} =  (s_{\lambda'}^{(k)} s_{l^{k+1-l}}) h_k = s_{\lambda' \cup l^{k+1-l}}^{(k)} h_k = s_{\lambda' \cup l^{k+1-l} \cup k}^{(k)} = s_{\lambda \cup l^{k+1-l}}^{(k)}.$$
 This completes the base of inner induction. Now let $\lambda_1 < \min(k,n)$, and again write $\lambda' = (\lambda_2,\ldots,\lambda_k)$. Since $|\lambda'| < |\lambda|$, we have
 $$s_{\lambda'}^{(k)} s_{l^{k+1-l}} = s_{\lambda' \cup l^{k+1-l}}^{(k)}$$
 by induction. Multiplication by $h_{\lambda_1}$ gives
 $$(s_{\lambda'}^{(k)} h_{\lambda_1}) s_{l^{k+1-l}} = s_{\lambda' \cup l^{k+1-l}}^{(k)}h_{\lambda_1},$$
 which yields, by Corollary \ref{pieri},
 \begin{equation} \label{eqn2} 
\left( \sum_{S \in \p S} s_{(\lambda')^S}^{(k)} \right) s_{l^{k+1-l}} = \sum_{S \in \p S'} s_{(\lambda' \cup l^{k+1-l})^S}^{(k)},
\end{equation}
 where $\p S$ (resp., $\p S'$) contains all $S \subseteq [k]$ of size ${\lambda_1}$ for which the entries $(i,j-1)$, $i < j$, $i \notin S$, $j \in S$, of the residue table of $\lambda'$ (resp., $\lambda' \cup l^{k+1-l}$) are $>0$. But $\lambda'$ and $\lambda' \cup l^{k+1-l}$ have the same residue tables, so $\p S = \p S'$. Clearly $[{\lambda_1}] \in \p S$, $(\lambda')^{[{\lambda_1}]} = \lambda$, $(\lambda' \cup l^{k+1-l})^{[{\lambda_1}]}=\lambda \cup l^{k+1-l}$, and if $S \neq [{\lambda_1}]$, then the largest part of $(\lambda')^S$ is $>{\lambda_1}$. By inner induction, the left-hand side of \eqref{eqn2} equals
 $$\left( s_\lambda^{(k)} + \!\!\!\sum_{[{\lambda_1}] \neq S \in \p S} \!\!\! s_{(\lambda')^S}^{(k)} \right) s_{l^{k+1-l}} = s_\lambda^{(k)} s_{l^{k+1-l}} + \!\!\!\sum_{[{\lambda_1}] \neq S \in \p S}\!\!\!s_{(\lambda')^S}^{(k)} s_{l^{k+1-l}} = s_\lambda^{(k)} s_{l^{k+1-l}} + \!\!\!\sum_{[{\lambda_1}] \neq S \in \p S}\!\!\!s_{(\lambda' \cup l^{k+1-l})^S}^{(k)},$$
 and clearly the right-hand side equals $s_{\lambda \cup l^{k+1-l}}^{(k)} + \sum_{[{\lambda_1}] \neq S \in \p S}s_{(\lambda' \cup l^{k+1-l})^S}^{(k)}$. After cancellations, we get $s_\lambda^{(k)} s_{l^{k+1-l}} = s_{\lambda \cup l^{k+1-l}}^{(k)}$.
\end{proof}

The theorem in particular implies that every $k$-Schur function can be written as the product of a $k$-Schur function corresponding to a $k$-irreducible partition, and Schur functions corresponding to rectangular partitions ${l^{k+1-l}}$.

\section{Murnaghan-Nakayama rule for $k$-Schur functions} \label{mn}

The Murnaghan-Nakayama rule has a generalization to $k$-Schur functions. There exists the concept of a \emph{$k$-ribbon} that plays the role of ribbons for Schur functions, in the sense that
$$s_\lambda^{(k)} p_n = \sum_\nu (-1)^{\hgt(\nu/\lambda)} s_\nu^{(k)},$$
for every $k$-bounded partition $\lambda$ and $n \leq k$, where the sum is over all $k$-bounded partitions $\nu$ for which $\nu/\lambda$ is a $k$-ribbon of size $n$, and $\hgt$ is an appropriately defined statistic for $k$-ribbons; see \cite[Corollary 1.4]{bsz}. The problem is that the definition of a $k$-ribbon (\cite[Definition 1.1]{bsz}) is extremely complicated; it involves not only the $k$-bounded partitions $\lambda$ and $\nu$ and the corresponding $(k+1)$-cores, but also the \emph{word} associated with ${\mf c^{(k)}}(\nu)/{\mf c^{(k)}}(\lambda)$ (which describes the contents of the cells added to ${\mf c^{(k)}}(\lambda)$ to obtain ${\mf c^{(k)}}(\nu)$).

\medskip

It turns out that residue tables again enable us to state the result in a unified and easy-to-check way. 

\medskip

Like in Theorems \ref{strong}, \ref{whs}, and \ref{wvs} (and see Remark \ref{remark}), we would like to answer the following question: given a multiset $S$ of size $n \leq k$, what conditions should a $k$-bounded partition $\lambda$ satisfy so that $s_{\lambda^S}^{(k)}$ appears in $s_\lambda^{(k)} p_n$, and with what coefficient? It turns out that the answer is the easiest when, for some $I$, we have
$$\sigma_1 \geq 1, \sigma_2 \geq 2, \sigma_3 \geq 2,\ldots,\sigma_I \geq 2, \sigma_{I+1} = \sigma_{I+2} = \ldots = 0.$$
Before we describe what data shows, let us remind the reader that the classical Murnaghan-Nakayama rule states that for such $S$, $s_{\lambda^S}$ appears in $s_\lambda p_n$ if and only if $m_i(\lambda) = \sigma_{i+1} - 1$ for $i = 1,\ldots,I-1$, and the coefficient is $(-1)^{n - I}$.

\medskip

Computer experimentation in the $k$-Schur function case shows the following. For $I = 1$ (when $S = \{1^{\sigma_1}\}$ and $\sigma_1 = n$), there is no condition to satisfy: for every $\lambda$, $s_{\lambda^S}^{(k)}$ appears in $s_\lambda^{(k)} p_n$ with coefficient $(-1)^{n-1}$. For $I = 2$, a $k$-irreducible partition $\lambda$ has to satisfy either $m_1(\lambda) \md k = \sigma_2-1$ (the ``classical'' answer) or $m_1(\lambda) \md k = k - \sigma_1$. For $I = 3$, a $k$-irreducible partition $\lambda$ has to satisfy one of six conditions (listed on the left in Table \ref{table1}), and for $I = 4$, one of $24$ conditions (six of which are listed on the right in the Table \ref{table1}).

\begin{table}[!ht]
{\small $$
\begin{array}{c|c||c|c|c}
 m_1(\lambda) \md k & m_2(\lambda) \md (k-1) &  m_1(\lambda) \md k& m_2(\lambda) \md (k-1) & m_3(\lambda) \md (k-2)\\ \hline
 \sigma_2-1 & \sigma_3-1 & \sigma_2-1 & \sigma_3-1 & \sigma_4-1 \\
 \sigma_2-1 & -\sigma_1-\sigma_2+k & \sigma_2+\sigma_3-1 & -\sigma_2+k-1 & -\sigma_1-\sigma_3+k-1 \\
 \sigma_2+\sigma_3-1 & -\sigma_2+k-1 & \sigma_2+\sigma_3+\sigma_4-1 & -\sigma_2-\sigma_4+k-1 & \sigma_4-1 \\
 -\sigma_1-\sigma_3+k & \sigma_3-1 & k-\sigma_1 & \sigma_1+\sigma_3-2 & \sigma_4-1 \\
 k-\sigma_1 & \sigma_1+\sigma_3-2 & -\sigma_1-\sigma_3+k & \sigma_3-1 & \sigma_1+\sigma_4-2 \\
 k-\sigma_1 & -\sigma_2+k-1 & -\sigma_1-\sigma_3-\sigma_4+k & \sigma_3+\sigma_4-1 & -\sigma_3+k-2
\end{array}$$}
\caption{$k$-ribbons in terms of $m_i(\lambda)$} \label{table1}
\end{table}

\medskip

The reader has probably guessed that there are $I!$ such conditions (compared to just one in the classical case!), but might be hard pressed to find a general pattern. A beautiful pattern emerges when we go to residue tables, however. Indeed, we have the following conjecture.

\begin{conjecture} \label{mn1}
 Given a multiset $S$ of size $n \leq k$ satisfying $\sigma_1 \geq 1, \sigma_i \geq 2$ for $2 \leq i \leq I$, $\sigma_{i} = 0$ for $i > I$, the coefficient of $s_{\lambda^S}^{(k)}$ in $s_\lambda^{(k)} p_n$ is nonzero if and only if for some permutation $\pi$ of $[I]$ we have, for $i = 2,\ldots,I$:
 \begin{align*}
  r_{1,i-1} &=\sum_{\three{j \in [I]}{j \neq 1}{\pi(j) \leq \pi(i)}} \sigma_j - \sum_{\three{j \in [i]}{j \neq 1}{\pi(j) \leq \pi(i)}} 1 & \mbox{if }& \pi(i) < \pi(1)\\
r_{1,i-1} &= k + 1 - \sigma_1 - \sum_{\three{j \in [I]}{j \neq 1}{\pi(j) > \pi(i)}} \sigma_j - \sum_{\three{j \in [i]}{j \neq 1}{\pi(j) \leq \pi(i)}} 1 & \mbox{if }& \pi(i) > \pi(1)
 \end{align*}
 Furthermore, when this condition is satisfied, the coefficient is independent of $\lambda$ and equals $(-1)^{n - I}$.
\end{conjecture}

\begin{example}
 For $I = 4$ and $\pi \in \set{4123,3214,4312,3412,2341,1243}$, we get the first rows of residue tables corresponding to partitions on the right-hand side of Table \ref{table1}. For example, for $\pi = 3214$, the conjecture gives conditions
 $$r_{11} = \sigma_2+\sigma_3 - 1, \qquad r_{12} = \sigma_3 - 1, \qquad r_{13} = k - \sigma_1 - 2,$$
 and for a ($k$-irreducible) partition $\lambda$ with $m_1(\lambda) \md k= \sigma_2+\sigma_3-1$, $m_2(\lambda) \md (k-1) = -\sigma_2+k-1$, $m_3(\lambda) \md (k-2)= -\sigma_1-\sigma_3+k-1$, we have
 \begin{align*} 
 r_{11} &= m_1(\lambda) \md k = \sigma_2+\sigma_3 - 1,\\
 r_{12} &= (m_2(\lambda) + r_{11}) \md (k-1) = (-\sigma_2+k-1+\sigma_2+\sigma_3-1) \md (k-1) = \sigma_3-1,\\
 r_{13} &= (m_3(\lambda) + r_{12}) \md (k-2) = (-\sigma_1-\sigma_3+k-1 + \sigma_3 - 1) \md (k-2) = k - \sigma_1 - 2.
 \end{align*}
\end{example}

Of course, since $m_i(\lambda)$ is, modulo $k+1-i$, equal to $r_{1i}-r_{1,i-1}$, we could rephrase Conjecture \ref{mn1} to give conditions on $m_i(\lambda)$ directly; these conditions would be only slightly more complicated that the ones given above. However, the approach via residue tables is more than justified when we observe what happens when we allow $\sigma_i = 1$ for $i = 2,\ldots,I$. Indeed, the conditions on $m_i(\lambda)$ become completely intractable, while the equalities on the first row of the residue table simply change to equalities and inequalities.

\begin{conjecture} \label{mn2}
 Given a multiset $S$ of size $n \leq k$ satisfying $\sigma_i \geq 1$ for $1 \leq i \leq I$, $\sigma_{i} = 0$ for $i > I$, the coefficient of $s_{\lambda^S}^{(k)}$ in $s_\lambda^{(k)} p_n$ is nonzero if and only if for some
 \begin{itemize}
  \item[(R1)] set $U \subseteq [I]$ satisfying $\sigma_i > 1 \Rightarrow i \in U$,
  \item[(R2)] permutation $\pi$ of $U$, and
  \item[(R3)] map $\varphi \colon [I] \setminus U \to U$ satisfying $\varphi(i) < i$ for all $i$,
 \end{itemize}
 we have, for $i = 2,\ldots,I$:
{\tiny
 \begin{align*}
  r_{1,i-1} &= \!\!\!\sum_{\three{j \in U}{j \neq 1}{\pi(j) \leq \pi(i)}}\!\!\! \sigma_j - \!\!\!\sum_{\three{j \in U \cap [i]}{j \neq 1}{\pi(j) \leq \pi(i)}}\!\!\! 1 - \!\!\!\!\!\sum_{\three{j \in [i-1] \setminus U}{\varphi(j) \neq 1}{\pi(\varphi(j)) \leq \pi(i)}}\!\!\!\!\! 1 & \mbox{if } & {\pi(i) < \pi(1)}\\
  r_{1,i-1} &= k + 1 - \sigma_1 - \!\!\!\!\!\sum_{\three{j \in U }{j \neq 1}{\pi(j) > \pi(i)}}\!\!\!\!\! \sigma_j - \!\!\!\!\!\sum_{\three{j \in U \cap [i] }{j \neq 1}{\pi(j) \leq \pi(i)}}\!\!\!\!\! 1 - \!\!\!\!\!\!\sum_{\three{j \in [i-1] \setminus U}{\varphi(j) \neq 1}{\pi(\varphi(j)) \leq \pi(i)}}\!\!\!\!\!\!\!\! 1  & \mbox{if } & {\pi(i) > \pi(1)}\\
   \!\!\!\sum_{\three{j \in U }{j \neq 1}{\pi(j) < \pi(\varphi(i))}}\!\!\! \sigma_j - \!\!\!\sum_{\three{j \in U \cap [i]}{j \neq 1}{\pi(j) \leq \pi(\varphi(i))}}\!\!\! 1 - \!\!\!\!\!\sum_{\three{j \in [i-1] \setminus U}{\varphi(j) \neq 1}{\pi(\varphi(j)) < \pi(\varphi(i))}}\!\!\!\!\! 1 <
  r_{1,i-1} &< \!\!\!\sum_{\three{j \in U }{j \neq 1}{\pi(j) \leq \pi(\varphi(i))}}\!\!\! \sigma_j - \!\!\!\sum_{\three{j \in U \cap [i] }{j \neq 1}{\pi(j) \leq \pi(\varphi(i))}}\!\!\! 1 - \!\!\!\!\!\sum_{\three{j \in [i-1] \setminus U}{\varphi(j) \neq 1}{\pi(\varphi(j)) \leq \pi(\varphi(i))}}\!\!\!\!\! 1 & \mbox{if } & {\pi(\varphi(i)) < \pi(1)}\\
  k+1 - \sigma_1 - \!\!\!\!\!\!\!\!\!\sum_{\three{j \in U }{j \neq 1}{\pi(j) \geq \pi(\varphi(i))}}\!\!\!\!\!\!\! \sigma_j - \!\!\!\!\!\!\!\sum_{\three{j \in U \cap [i]}{j \neq 1}{\pi(j) \leq \pi(\varphi(i))}}\!\!\!\!\!\!\! 1 - \!\!\!\!\!\!\sum_{\three{j \in [i-1] \setminus U}{\varphi(j) \neq 1}{\pi(\varphi(j)) < \pi(\varphi(i))}}\!\!\!\!\!\!\!\!\!\!\! 1 <
  r_{1,i-1} &< k+1 - \sigma_1 - \!\!\!\!\!\!\!\!\!\sum_{\three{j \in U }{j \neq 1}{\pi(j) > \pi(\varphi(i))}}\!\!\!\!\!\!\! \sigma_j - \!\!\!\!\!\!\!\sum_{\three{j \in U \cap [i] }{j \neq 1}{\pi(j) \leq \pi(\varphi(i))}}\!\!\!\!\!\!\! 1 - \!\!\!\!\!\!\sum_{\three{j \in [i-1] \setminus U}{\varphi(j) \neq 1}{\pi(\varphi(j)) \leq \pi(\varphi(i))}}\!\!\!\!\!\!\!\!\!\!\! 1 & \mbox{if } & {\pi(\varphi(i)) > \pi(1)}\\
  k+1 - \sigma_1 - \!\!\!\sum_{j \in U \cap [i]}\!\!\! 1 - \!\!\!\!\!\sum_{\genfrac{}{}{0pt}{}{j \in [i-1] \setminus U}{\varphi(j) \neq 1}}\!\!\!\!\! 1 <
  r_{1,i-1} &  & \mbox{if } & { \varphi(i) = 1}
 \end{align*}}
  Furthermore, when this condition is satisfied, the coefficient is independent of $\lambda$ and equals $(-1)^{n - I}$.
\end{conjecture}

\begin{remark}
 Note that $\pi(i)$ only makes sense if $i \in U$, and $\varphi(i)$ (and $\pi(\varphi(i))$) only if $i \notin U$. Also note that (R3) implies that $1 \in U$.\\
 If $\sigma_i \geq 2$ for $i = 2,\ldots,I$, then $U = [I]$ and $\varphi$ is the empty map, so Conjecture \ref{mn2} is a generalization of Conjecture \ref{mn1}.
\end{remark}

\begin{example}
 Take $S = \set{\sigma_1,\sigma_2,\sigma_3,1,\sigma_5}$ with $\sigma_2,\sigma_3,\sigma_5 \geq 2$. For $U$, we can choose either $U = \set{1,2,3,4,5}$ (which gives the same conditions as Conjecture \ref{mn1}) or $U = \set{1,2,3,5}$. For, say, $U = \set{1,2,3,5}$, $\pi = 2135$, $\varphi(4) = 3$, we get
 $$r_{11} = \sigma_2-1, \quad r_{12} = k-\sigma_1-\sigma_5-1, \quad k-\sigma_1-\sigma_3-\sigma_5-1< r_{13} < k-\sigma_1-\sigma_5-1, \quad r_{14} = k-\sigma_1-3.$$
\end{example}

The conjectures were checked with a computer for all $k$-bounded partitions for $k$ up to $9$ and for all $n \leq k$.

\medskip

Of course, there are two more crucial steps that need to be done before one can truly say that residue tables are the right way to describe $k$-ribbons. First, we would need to describe the conditions for an arbitrary $S$, \textit{i.e.} one that can have $\sigma_i = 0$, $\sigma_j > 0$ for $i < j$; some preliminary work has been done in this direction and it certainly seems feasible. And secondly, the resulting conjecture(s) would have to be proved. We leave this as an open problem.

\section{Further conjectures and open problems} \label{further}

\subsection{Multiplication with an almost-$k$-rectangle and quantum Monk's formula} \label{almost}

The coefficients in the expansion of a product of a $k$-Schur function with a $k$-Schur function corresponding to a the partition $l^{k+1-l}$ with a cell removed are also either $0$ or $1$. The following was stated as a conjecture by the author and proved by Luc Lapointe and Jennifer Morse (private communication; the proof is presented here with permission).

\begin{proposition} \label{jenn}
 Let $\nu$ denote the partition with $k-l$ copies of $l$ and one copy of $l-1$ for $1 \leq l \leq k$, and let $\lambda$ be a $k$-bounded partition with residue table $R = (r_{ij})_{1 \leq i \leq j \leq k}$. Then $s_\lambda^{(k)} s_\nu = \sum_{I,J} s_{\lambda^{S(I,J)}}^{(k)}$, where the sum is over all $I,J$, $I \leq l$, $J > l$, for which $r = r_{I,J-1}$ is strictly smaller than all entries to its left in $R$, and $S(I,J)$ is the multiset $\set{1^{k+1-l},\ldots,(I-1)^{k+1-l},I^{k-l-r},(I+1)^{k+1-l},\ldots,l^{k+1-l},J^r }$.
\end{proposition}





The importance of the proposition is that it can be used to to deduce the quantum Monk's formula via the isomorphism of \cite{ls} that identifies quantum Schubert polynomials and k-Schur functions. See \cite{monk}.

\begin{example}
 Let us compute $s^{(4)}_{32211} \cdot s_{221}$ by using the proposition for $k = 4$, $l = 2$, $\lambda = 32211$. The residue table of $\lambda$ is $\begin{smallmatrix}
    2 & \underline{1} & \underline 0 & 0 \\
   & \underline 2 & \underline 1 & \underline 0 \\
   &   & 1 & 0 \\
   &   &   & 0
   \end{smallmatrix}$; the entries in rows up to $l  = 2$ and in columns from $l = 2$ onwards which are strictly smaller than the entries to their left	 are underlined. Each such entry yields one term in $s^{(4)}_{32211} \cdot s_{221}$, and since $\lambda^{S(1,3)} = \lambda^{\set{1,2^3,3}} = 332222$, $\lambda^{S(1,4)} = \lambda^{\set{1^2,2^3}} = 3222221$, $\lambda^{S(2,3)} = \lambda^{\set{1^3,3^2}} = 33311111$, $\lambda^{S(2,4)} = \lambda^{\set{1^3,2,4}} = 42221111$, $\lambda^{S(2,5)} = \lambda^{\set{1^3,2^2}} = 32222111$, this means that $s^{(4)}_{32211} \cdot s_{221}$ equals
  $$s^{(4)}_{332222}+s^{(4)}_{3222221}+s^{(4)}_{33311111}+s^{(4)}_{42221111}+s^{(4)}_{32222111}.$$
\end{example}

\begin{proof}[Proof of Proposition \ref{jenn}]
 Let $\langle \cdot , \cdot \rangle_k$ be the scalar product with respect to which the $k$-Schur functions are dual to the dual $k$-Schur functions (see \cite[\S 2.2.2]{llmssz}). Let $p_1^\perp$ be such that $\langle p_1^\perp f,g \rangle_k = \langle f, p_1 g \rangle_k$. It easily follows from the definition and \cite[Proposition 7.9.3]{stanley} (which also holds for $\langle \cdot , \cdot \rangle_k$) that $p_1^\perp p_\lambda = m_1(\lambda) p_{\bar \lambda}$, where $\bar \lambda$ is $\lambda$ with one copy of $1$ removed (if $m_1(\lambda) = 0$, we interpret this as $p_1^\perp p_\lambda = 0$). Therefore $p_1^\perp (p_\lambda p_\mu) = (p_1^\perp p_\lambda) p_\mu + p_\lambda (p_1^\perp p_\mu)$, and so $p_1^\perp (fg) = (p_1^\perp f)g+f(p_1^\perp g)$ for all $f,g$.\\
 It immediately follows from duality and \cite[Theorem 2.2.22]{llmssz} that $p_1^\perp(s_\lambda^{(k)})$ is the sum of all $s_\mu^{(k)}$'s (with multiplicities) such that $\lambda$ is a strong cover of $\mu$. For example, $p_1^\perp(s_{l^{k+1-l}}) = s_\nu$ (in the notation of the proposition), as the residue table of $l^{k+1-l}$ is all zeros, the only elements of the residue table that are strictly smaller than all the elements to the left are the ones on the diagonal, and only the one in position $(l,l)$ has a non-zero multiplicity (equal to $q_{ll} = 1$); see Theorem \ref{strong}. That means that
 $$p_1^\perp (s_{l^{k+1-l}} s_\lambda^{(k)}) = p_1^\perp(s_{l^{k+1-l}}) s_\lambda^{(k)} + s_{l^{k+1-l}} p_1^\perp(s_\lambda^{(k)}) = s_\nu s_\lambda^{(k)} + s_{l^{k+1-l}} p_1^\perp(s_\lambda^{(k)}),$$
 and
 $$s_\nu s_\lambda^{(k)} = p_1^\perp(s_{\lambda \cup l^{k+1-l}}^{(k)}) - s_{l^{k+1-l}} p_1^\perp(s_\lambda^{(k)}).$$
 Obviously, the residue tables of $\lambda$ and $\lambda \cup l^{k+1-l}$ are equal, and the quotient table of $\lambda^{k+1-l}$ is that of $\lambda$ with $1$ added to every entry in column $l$. Now take a pair $(I,J)$ so that $r_{I,J-1}$ is smaller than all the elements to its left in the residue table $R = (r_{ij})$ of $\lambda$ (or $\lambda \cup l^{k+1-l}$). If $I > l$ or $J \leq l$, the elements in positions $(I,I),\ldots,(I,J-1)$ in the quotient tables of $\lambda$ and $\lambda \cup l^{k+1-l}$ are equal, so the corresponding $k$-Schur functions cancel out. If $I \leq l$ and $J > l$, the elements in positions $(I,I),\ldots,(I,l-1),(I,l+1),\ldots,(I,J-1)$ of the quotient tables are equal, and the element in position $(I,l)$ is greater by $1$ in the quotient table of $\lambda \cup l^{k+1-l}$. After cancellations, we get precisely one copy of the corresponding $s_\nu^{(k)}$.
\end{proof}

\begin{remark}
 The proof of the proposition illustrates that ``$\lambda$ covers $\mu$ in the strong order with multiplicity $0$'' is fundamentally different from ``$\lambda$ does not cover $\mu$ in the strong order''. Indeed, when $\lambda$ covers $\mu$ in the strong order with multiplicity $0$, $\lambda \cup l^{k+1-l}$ covers $\mu \cup l^{k+1-l}$ in the strong order with multiplicity $1$ for an appropriate $l$, which was crucial in the proof. In other words, it is fundamentally better to describe strong covers in terms of residue and quotient tables than in terms of cores.
\end{remark}

\begin{problem}
 Describe the conditions under which $s^{(k)}_{\lambda^S}$ appears in $s^{(k)}_\lambda s_\mu$, where $\mu$ is a partition with $k-l$ copies of $l$ and one copy of $l'$ for $1 \leq l \leq k$ and $0 \leq l' \leq l$.
\end{problem}

For $l' = l$, this is \eqref{eqn}, and for $l' = l-1$, this is Proposition \ref{jenn}.

\subsection{Splitting of $k$-bounded partitions} \label{split}

For a $k$-bounded partition $\lambda$, denote by $\partial_k(\lambda)$ the cells of $\mf c^{(k)}(\lambda)$ with hook-length $\leq k$. If $\partial_k(\lambda)$ is not connected, we say
that $\lambda$ \emph{splits}. Each of the connected components of $\partial_k(\lambda)$ is a horizontal translate of $\partial_k(\lambda^i)$ for some $k$-bounded partition $\lambda^i$. Call $\lambda^1,\ldots,\lambda^I$ the \emph{components} of $\lambda$. It is easy to see that if $\lambda$ splits into  $\lambda^1,\ldots,\lambda^I$, there must be $J_1,\ldots,J_I$ so that $\lambda^i = \langle (J_{i-1}+1)^{m_{J_{i-1}+1}(\lambda)},\ldots,J_i^{m_{J_i}(\lambda)}\rangle$, \textit{i.e.}, for each $i$, all copies of $i$ correspond to the same component. We say that $\lambda$ \emph{splits at $J_1,\ldots,J_I$}.

\begin{example}
 Figure \ref{fig:split} depicts $\partial_5(54433211)$.

\begin{figure}[!ht]
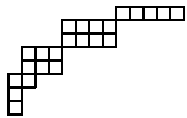

 \begin{center}

\ytableausetup{boxsize=0.4em}
  \begin{center}
 \ydiagram{8+5,4+4,4+4,1+3,1+3,2,1,1}
 \end{center}
\ytableausetup{nosmalltableaux}

\end{center}
 \caption{Splitting of a $k$-partition.} \label{fig:split}
\end{figure}
 It follows that $\lambda$ splits into components $\lambda^1=33211,\lambda^2=44,\lambda^3=5$. In other words, $\lambda$ splits at $3$, $4$, and $5$.

\end{example}

Denton \cite[Theorem 1.1]{denton} proved the following theorem.

\begin{theorem} \label{thm:split}
 Suppose $\lambda$ splits into $\lambda^1,\ldots,\lambda^I$. Then $s^{(k)}_\lambda = \prod_{i = 1}^I s^{(k)}_{\lambda^i}$. \qed
\end{theorem}

The following is easy to prove.

\begin{proposition}\label{prop:split}
 A $k$-bounded partition $\lambda$ splits at $J$ if and only if all entries of column $J$ of the residue table of $\lambda$ contains are equal to either $0$ or $k-J$. \qed
\end{proposition}

The consequence of $r_{i,J} \in \set{0,k-J}$ for $i = 1,\ldots,J$ is that $r_{i,J+1} = r_{J+1,J+1}$, $r_{i,J+2} = r_{J+2,J+2}$, \textit{etc.} In other words, the residue table of a partition that splits at $J$ is composed of two essentially independent parts, the triangles $(r_{ij})_{1 \leq i \leq j \leq J-1}$ and $(r_{ij})_{J \leq i \leq j \leq k}$; the rectangle $(r_{ij})_{1 \leq i \leq J-1, J \leq j \leq k}$ is composed of $J-1$ copies of the line $(r_{iJ})_{J \leq i \leq k}$. So Denton's theorem states that when the residue table is composed of such independent parts (each a residue table itself!), the corresponding $k$-Schur function is the product of $k$-Schur functions of the parts.

\begin{problem}
 Find a simpler proof of Theorem \ref{thm:split}, using Proposition \ref{prop:split} and the preceding paragraph.
\end{problem}

\subsection{LLMS insertion} \label{llms}

In \cite{LLMS10}, a variant of the Robinson-Schen\-sted insertion for strong marked and weak tableaux is presented. More specifically, starting with a square integer matrix $M$ of size $m \times m$, they construct a \emph{growth diagram} of $M$, an $(m+1) \times (m+1)$ grid that has the empty partition in every vertex in the top row and left-most column, the integer $m_{ij}$ in the square between vertices $(i,j)$, $(i,j+1)$, $(i+1,j)$ and $(i+1,j+1)$, a strong marked horizontal strip (we omit the definition) on horizontal edges, and a weak horizontal strip on vertical edges; like in the classical case, the new weak and strong strips are constructed using certain local rules, but they are extremely complicated. The procedure has very important implications and it would be important to understand it better.

\medskip

The local rules simplify slightly when specialized to standard tableaux \cite[\S 10.4]{LLMS10} (\textit{i.e.}, when $M$ is a $0/1$ matrix). The following description of case X (external insertion) hints that a description in terms of residue tables could be possible. 

\begin{conjecture}
 Suppose that all three known corners of a square in the growth diagram are the same, say $\lambda$, and the number within the square is $1$. Then the unknown corner of the square is $\lambda^{\set i}$, where $i$, $1 \leq i \leq k$, is the unique index for which $r_{j,i-1} > 0$ for $1 \leq j \leq i-1$ and $r_{ij} < k - j$ for $i \leq j \leq k-1$.
\end{conjecture}

\begin{example}
 For $k = 2$ and $\lambda = \emptyset$ (resp., $\lambda = 1$, $\lambda = 11$, $\lambda = 111$), the residue table is $\begin{smallmatrix}0 & 0 \\  & 0\end{smallmatrix}$ (resp., $\begin{smallmatrix}1 & 1 \\  & 0\end{smallmatrix}$, $\begin{smallmatrix}0 & 0 \\  & 0\end{smallmatrix}$, $\begin{smallmatrix}1 & 1 \\  & 0\end{smallmatrix}$), and the unique $i$ satisfying the condition from the conjecture is $1$ (resp., $2$, $1$, $2$). That means that the bottom right corner of the square of the growth diagram is $\emptyset^{\set 1} = 1$ (resp., $1^{\set 2} = 2$, $11^{\set 1} = 111$, $111^{\set 2} = 211$), as confirmed by the example in \cite[\S 10.4]{LLMS10}.
\end{example}

The following justifies the term ``unique'' in the conjecture.

\begin{proposition}
 For a $k$-bounded partition $\lambda$, there exists exactly one $i$, $1 \leq i \leq k$, for which $r_{j,i-1} > 0$ for $1 \leq j \leq i-1$ and $r_{ij} < k - j$ for $i \leq j \leq k-1$.
\end{proposition}
\begin{proof}
 Let $i'$ be the minimal $i$ for which $r_{ij} < k - j$ for $i \leq j \leq k-1$ ($i = k$ satisfies this condition, so $i'$ is well defined). We prove that $r_{j,i'-1} > 0$ for $1 \leq j \leq i'-1$ by contradiction. Assume that $r_{j,i'-1} = 0$ for some $j$, $1 \leq j \leq i' - 1$, let $j'$ be the largest such $j$. Since $r_{j',i'-1} = 0$, row $j'$ from position $i'$ onwards is the same as row $i'$; therefore $r_{j'j} < k - j$ for $i' \leq j \leq k-1$. By minimality of $i'$, there must be $j$, $j' \leq j < i'-1$, for which $r_{j'j} = k - j$, let $j''$ be the largest such $j$. Since $r_{j'j''} = k-j''$, row $j'$ from position $j''+1$ onwards is the same as row $j''+1$. That implies that $r_{j'' + 1,j} < k - j$ for $j''+ 1 \leq j \leq k - 1$, and since $j'' + 1 < i'$, this contradicts the minimality of $i'$.\\
 Now assume that both $i'$ and $i''$, $i' < i''$, satisfy the conditions. Since $r_{i',i''-1} > 0$ (by the conditions for $i''$) and $r'_{i',i''-1} < k + 1- i''$ (by the conditions for $i'$), we have $r_{i'i''} \neq r_{i''i''}$. But $r_{i'k} = r_{i''k}$ ($ = 0$), so there must be $j$, $i'' \leq j < k$, so that $r_{i'j} \neq r_{i''j}$ and $r_{i',j+1} = r_{i'',j+1}$. Since $r_{i',j+1} = (r_{i'j} + m_{j+1}(\lambda)) \md (k-j) = r_{i'',j+1} = (r_{i''j} + m_{j+1}(\lambda)) \md (k-j)$, we have $r_{i'j} \md (k-j) = r_{i''j} \md (k-j)$, which can only be true if one of $r_{i'j}$ and $r_{i''j}$ is $0$ and the other one is $k-j$, which contradicts $r_{i'j} < k - j$ and $r_{i''j} < k - j$.
\end{proof}

\begin{problem}
 Describe cases A and B of LLMS insertion in the standard case in terms of residue tables. Describe cases A, B, C and X of the general LLMS insertion in terms of residue tables.
\end{problem}

\section{On multiplication of arbitrary $k$-Schur functions} \label{kLR}

\subsection{Classical Littlewood-Richardson rule} \label{lr}

A description of the coefficients in the expansion of a product of two Schur functions in terms of Schur functions, the \emph{Littlewood-Richardson rule}, is one of the major results of classical symmetric function theory. The theorem has many versions; we will need (and later slightly adapt) the following.

\medskip

Recall that we call the coefficients $c_{\lambda\mu}^\nu$ in the expansion $s_\lambda s_\mu = \sum_\nu c_{\lambda\mu}^\nu s_\nu$ \emph{Littlewood-Richard\-son coefficients}. For a (skew) semistandard Young tableaux $T$ (a map from the cells of the diagram of $\lambda$ to $\set{1,2,\ldots}$), take the numbers in $T$ from top to bottom, right to left; the resulting word is called the \emph{reverse reading word of $T$}. A word $a = a_1a_2\ldots a_m$ with $a_i \in \N$ is a \emph{lattice permutation} if the number of $i$'s in $a_1a_2\ldots a_j$ is greater than or equal to the number of $(i+1)$'s for all $i$ and $j$.

\begin{theorem}{\cite[Theorem A1.3.3]{stanley}}
 For partitions $\lambda,\mu,\nu$, the coefficient $c_{\lambda\mu}^\nu$ is equal to the number of semistandard Young tableaux of shape $\nu/\lambda$ and weight $\mu$ whose reverse reading word is a lattice permutation. \qed
\end{theorem}

\begin{example}
 Take $\lambda = 32$, $\mu = 21$ and $\nu = 431$. The semistandard Young tableaux of shape $431/32$ and weight $21$ are shown in Figure \ref{fig4}.
 
\begin{figure}[!ht]
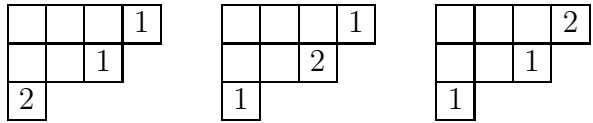
 
\begin{center}
 \ytableausetup{boxsize=1.2em}
$$\begin{ytableau}
  \phantom{1} & \phantom{1}& \phantom{1}& 1 \\ 
\phantom{1}& \phantom{1}& 1 \\ 
2
  \end{ytableau} \qquad 
\begin{ytableau}
  \phantom{1} & \phantom{1}& \phantom{1}& 1 \\ 
\phantom{1}& \phantom{1}& 2 \\ 
1
  \end{ytableau} \qquad 
\begin{ytableau}
  \phantom{1} & \phantom{1}& \phantom{1}& 2 \\ 
\phantom{1}& \phantom{1}& 1 \\ 
1
  \end{ytableau} \quad $$
\caption{The computation of $c_{32,21}^{431} = 2$.}\label{fig4}
\end{center}
\end{figure}

The reverse reading words are $112$, $121$ and $211$, respectively. Since only the first two are lattice permutations, we have $c_{32,21}^{431} = 2$.
\end{example}

There is another way to think about this, one that will prove crucial for our purposes (compare with Remark \ref{remark}). Define an \emph{array} $\p A = (A_1,A_2,\ldots)$ to be a sequence of (finite) sets of positive integers, with all but finitely many of them empty. The \emph{shape} of an array $\p A$ is the sequence $(|A_1|,|A_2|,\ldots)$. The \emph{weight} of an array $\p A$ is the composition $(m_1,m_2,\ldots)$, where $m_i$ is the total number of $i$'s in $A_1,A_2,\ldots$. For example, $\p A = (\set{2},\emptyset,\set{1},\set{1},\emptyset,\emptyset,\ldots)$ and $\p B = (\emptyset,\set{1,2},\set{1},\emptyset,\emptyset,\ldots)$ are arrays of shapes $1011$ and $021$ and weight $21$. We will usually present the sets $A_i$ as columns of increasing numbers and the empty sets as dots, and leave out the trailing empty sets; so $\p A = 2 \cdot 1 1$, $\p B = \begin{matrix}
\cdot & 1 & 1 \\
   & 2 & 
\end{matrix}$. We write $a_{ij}$ for the $j$-th largest integer in $A_i$, with $a_{ij} = \infty$ if $j > |A_i|$. 
If $\lambda \subseteq \nu$, call the sequence $(\nu'_i-\lambda'_i)_{i \geq 1}$ the \emph{shape} of $\nu/\lambda$. For example, the shape of $431/32$ is $1011$.

\medskip

We can identify every semistandard Young tableau of shape $\nu/\lambda$ and weight $\mu$ with the addition of an array of shape $\nu/\lambda$ and weight $\mu$ to $\lambda$. The previous example shows the addition of arrays $2 \cdot 1 1$, $1 \cdot 2 1$ and $1 \cdot 1 2$ to $32$. 

\medskip

On the other hand, we cannot add \emph{any} array to any shape. For example, if we add $2 \cdot 1 1$ to a partition $\lambda$ with $m_2(\lambda) = 0$, the $1$ in the third column will be to the right of an empty square, which is impossible. Indeed, it is easy to see that we can add $\p A = 2 \cdot 1 1$ if and only if $m_2(\lambda) \geq 1$, and we can add $\p B = \begin{matrix}
\cdot & 1 & 1 \\
   & 2 & 
\end{matrix}$ if and only if $m_1(\lambda) \geq 2$.

\medskip

Write $\lambda^{\p A}$ for the tableau obtained if $\p A$ is added to $\lambda$. We have the following lemma.

\begin{lemma} \label{lemma1}
 The tableau $\lambda^{\p A}/\lambda$ is semistandard if and only if $m_{i}(\lambda) \geq c_i$ for all $i$, where $c_i$ is the minimal among nonnegative integers $c$ for which $a_{i,j-c} \leq a_{i+1,j}$ for all $j$, $c < j \leq |A_{i+1}|$. \qed
\end{lemma}

While the statement may seem complicated, it is just saying that we need to ``push down'' a column $A_i$ until the entries are to the left of larger or equal entries in column $A_{i+1}$. We say that \emph{$\lambda$ satisfies the LR condition for $\p A$} if the conditions of Lemma \ref{lemma1} are satisfied.

\begin{example}
 For the array $\p A = (\set 1, \set{2,3}, \emptyset, \set{1,4}, \set{1,2,3},\ldots)$,
 we have that $\lambda^{\p A}/\lambda$ is a semistandard Young tableau if and only if $m_1(\lambda) \geq 1$, $m_3(\lambda) \geq 2$ and $m_4(\lambda) \geq 2$. See Figure \ref{fig5}.

\begin{figure}[!ht] 
\begin{center}
 \ytableausetup{boxsize=1.2em}
\begin{ytableau}
{} & & & & 1 \\
& & & & 2 \\
& & & 1 & 3 \\
& & & 4 \\
& 2 \\
1 & 3
  \end{ytableau}
\caption{An illustration of LR conditions for an array.}\label{fig5}
\end{center}
\end{figure}
\end{example}

Of course, the semistandard Young tableaux that appear in Littlewood-Richardson rule have to satisfy an additional property, namely, the reverse reading word has to be a lattice permutation. Our crucial observation is the fact that whether or not the reverse reading word of the semistandard Young tableau $\lambda^{\p A}/\lambda$ is a lattice permutation or not depends only on $\p A$. We call an array $\p A$ an \emph{LR array} if the total number of $i$'s in sets $A_j, A_{j+1},\ldots$ is greater than or equal to the total number of $(i+1)$'s in these sets, for all $i$ and $j$. The array $\p A$ from the last example is an LR array, while $1 \cdot 1 2$ is not. We have the following result.

\begin{lemma} \label{lemma2}
 The reverse reading word of $\lambda^{\p A}/\lambda$ (for every $\lambda$ that satisfies the LR condition for $\p A$) is a lattice permutation if and only if $\p A$ is an LR array. \qed
\end{lemma}

Lemmas \ref{lemma1} and \ref{lemma2} allow us to formulate the following version of the Littlewood-Richardson rule.

\begin{theorem}
 The Littlewood-Richardson coefficient $c_{\lambda\mu}^\nu$ is equal to the number of LR arrays with the same shape as $\nu/\lambda$ and weight $\mu$ whose LR conditions $\lambda$ satisfies. \qed
\end{theorem}

\begin{example}
 As noted before (in different words), there are three arrays of shape $1011$ and weight $21$, namely $2 \cdot 1 1$, $1 \cdot 2 1$ and $1 \cdot 1 2$, but only the first two are LR arrays. Since the LR conditions for $2 \cdot 1 1$ (resp., $1 \cdot 2 1$) are $m_2(\lambda) \geq 1$ (resp., $m_2(\lambda) \geq 1$, $m_3(\lambda) \geq 1$), then for $\nu$ with $\nu-\lambda=1011$ we have $c_{\lambda,21}^\nu \in \set{0,1,2}$, depending on whether neither, one, or both of these sets of conditions are satisfied. For example, since $\lambda = 32$ satisfies conditions for both $2 \cdot 1 1$ and $1 \cdot 2 1$ and the shape of $431/32$ is $1011$, we have $c_{32,21}^{431} = 2$.
\end{example}

The author readily admits that this is a very complicated way to state the Littlewood-Richardson rule. However, it seems possible that the Littlewood-Richardson rule for $k$-Schur functions could be stated in a similar way. 

\subsection{Toward a $k$-Littlewood-Richardson rule} \label{toward}

Define the \emph{$k$-Littlewood-Richardson coefficients} $c_{\lambda\mu}^{(k),\nu}$ by the formula 
$$s^{(k)}_\lambda s^{(k)}_\mu = \sum_\nu c_{\lambda\mu}^{(k),\nu} s^{(k)}_\nu$$
for $k$-bounded partitions $\lambda,\mu$, where the sum is over all $k$-bounded partitions $\nu$. Then there appear to exist \emph{$k$-LR conditions}, (relatively) simple conditions on the residue table of $\lambda$, similar to conditions in Theorems \ref{whs} and \ref{wvs}, so that $c_{\lambda\mu}^{(k),\nu}$ is the number of LR arrays of the same shape as $\nu/\lambda$ and weight $\lambda$ for which these conditions are satisfied. 

\medskip

The exact form of some of these conditions is the content of this subsection.

\medskip

Note that in particular we conjecture that $c_{\lambda\mu}^{(k),\nu}$ is bounded from above by the number of LR arrays of the same shape as $\nu/\lambda$ and weight $\lambda$, which is, of course, not at all obvious.

\medskip

Let us first restate Theorems \ref{whs} and \ref{wvs} in the language of $k$-Littlewood-Richardson coefficients. 

\medskip

Recall the natural analogues of the Pieri rule and the conjugate Pieri rule for $k$-Schur functions, \eqref{kpieri} and \eqref{kconjpieri}.

\medskip

There is clearly exactly one array of shape $s$ and weight $n$ if $s$ contains $n$ ones and the rest are zeros (and it is clearly also an LR array). Also, there is exactly one LR array of shape $s$ and weight $1^n$, as long as the size of $s$ is $n$. For example, the only LR array of shape $101101$ and weight $4$ is $1 \cdot 1 1 \cdot 1$, and the only LR array of shape $301201$ and weight $1^7$ is 
$$\begin{matrix}
5 & \cdot & 4 & 2 & \cdot & 1 \\
6 & & & 3 & & \\
7 & & & & & 
\end{matrix}$$

Based on Theorem \ref{whs}, it makes sense to define the {$k$-LR conditions} for an array $\p A = (A_1,A_2,\ldots)$ with $A_i$ either $\set 1$ or $\emptyset$ as $r_{i,j-1}(\lambda) > 0$ for $A_i = \emptyset$, $A_j = \set 1$.

\medskip

Because of Theorem \ref{wvs}, it makes sense to define the {$k$-LR conditions} for an array $\p A = (A_1,A_2,\ldots)$ for which the numbers in each $A_i$ are consecutive, every number in $A_i$ is larger than any number in $A_j$ for $i < j$, and $\bigcup A_i = [n]$ for some $n$, as follows:
$$|A_j| \leq r_{i,j-1} \leq k + 1 - j - |A_i| + |\{ h \colon i \leq h \leq j-2, r_{ih} > k-h-|A_i|, r_{ih} > r_{i,j-1} \} | $$
 for $i < j$, $A_j \neq \emptyset$. This allows us to state the following.

\medskip

The $k$-Littlewood-Richardson coefficient $c_{\lambda, n}^{(k),\nu}$ (resp., $c_{\lambda,1^n}^{(k),\nu}$) is equal to the number of LR arrays of shape $\nu-\lambda$ and weight $n$ (resp., $1^n$) whose $k$-LR conditions $\lambda$ satisfies.

\medskip

These two sets of conditions have something in common. Namely, given an LR array $\p A$ of weight $\mu$, we have some lower and/or upper bounds on $r_{i,j-1}$ for each $i,j$, $i < j$, based on what $A_i$ and $A_j$ are (note that these bounds can involve other $r_{ih}$ for $i \leq h \leq j-2$).

\medskip

One would hope that such conditions exist for all LR arrays. As in Section \ref{mn}, it makes sense to first guess such conditions in the simplest case, namely when the array consists of only $1$'s and $2$'s. The main difficulty in guessing the correct conditions is that while the $k$-Littlewood-Richardson coefficients $c_{\lambda, n}^{(k),\nu}$ and $c_{\lambda,1^n}^{(k),\nu}$ are always $0$ or $1$, this does not, of course, hold for $c_{\lambda \mu}^{(k),\nu}$ for arbitrary $\mu$. To illustrate, take $\mu = 22$. There are two LR arrays of shape $1111$ and weight $22$, namely, $2211$ and $2121$, and indeed $c_{\lambda,22}^{(k),\nu}$ is always either $0$, $1$ or $2$ when $\nu/\lambda$ has shape $1111$ (for all $k$ checked). Among the $120$ $5$-irreducible partitions, the desired $k$-Littlewood-Richardson coefficients are all $0$ or $1$; the following shows the residue tables of all $24$ of them with the coefficient equal to $1$.

$$
\begin{smallmatrix}
 0 & 1 & 1 & 1 & 0 \\
   & 1 & 1 & 1 & 0 \\
   &   & 0 & 0 & 0 \\
   &   &   & 0 & 0 \\
   &   &   &   & 0
\end{smallmatrix}\quad
\begin{smallmatrix}
 0 & 1 & 1 & 0 & 0 \\
   & 1 & 1 & 0 & 0 \\
   &   & 0 & 1 & 0 \\
   &   &   & 1 & 0 \\
   &   &   &   & 0
\end{smallmatrix}\quad
\begin{smallmatrix}
 0 & 2 & 1 & 1 & 0 \\
   & 2 & 1 & 1 & 0 \\
   &   & 2 & 0 & 0 \\
   &   &   & 0 & 0 \\
   &   &   &   & 0
\end{smallmatrix}\quad
\begin{smallmatrix}
 0 & 2 & 1 & 0 & 0 \\
   & 2 & 1 & 0 & 0 \\
   &   & 2 & 1 & 0 \\
   &   &   & 1 & 0 \\
   &   &   &   & 0
\end{smallmatrix}\quad
\begin{smallmatrix}
 1 & 1 & 2 & 0 & 0 \\
   & 0 & 1 & 1 & 0 \\
   &   & 1 & 1 & 0 \\
   &   &   & 0 & 0 \\
   &   &   &   & 0
\end{smallmatrix}\quad
\begin{smallmatrix}
 1 & 1 & 2 & 1 & 0 \\
   & 0 & 1 & 0 & 0 \\
   &   & 1 & 0 & 0 \\
   &   &   & 1 & 0 \\
   &   &   &   & 0
\end{smallmatrix}\quad
\begin{smallmatrix}
 1 & 3 & 1 & 1 & 0 \\
   & 2 & 0 & 0 & 0 \\
   &   & 1 & 1 & 0 \\
   &   &   & 0 & 0 \\
   &   &   &   & 0
\end{smallmatrix}\quad
\begin{smallmatrix}
 1 & 3 & 1 & 0 & 0 \\
   & 2 & 0 & 1 & 0 \\
   &   & 1 & 0 & 0 \\
   &   &   & 1 & 0 \\
   &   &   &   & 0
\end{smallmatrix}
$$
$$
\begin{smallmatrix}
 2 & 2 & 0 & 0 & 0 \\
   & 0 & 1 & 1 & 0 \\
   &   & 1 & 1 & 0 \\
   &   &   & 0 & 0 \\
   &   &   &   & 0
\end{smallmatrix} \quad
\begin{smallmatrix}
 2 & 2 & 0 & 1 & 0 \\
   & 0 & 1 & 0 & 0 \\
   &   & 1 & 0 & 0 \\
   &   &   & 1 & 0 \\
   &   &   &   & 0
\end{smallmatrix} \quad
\begin{smallmatrix}
 2 & 3 & 1 & 1 & 0 \\
   & 1 & 2 & 0 & 0 \\
   &   & 1 & 1 & 0 \\
   &   &   & 0 & 0 \\
   &   &   &   & 0
\end{smallmatrix} \quad
\begin{smallmatrix}
 2 & 3 & 1 & 0 & 0 \\
   & 1 & 2 & 1 & 0 \\
   &   & 1 & 0 & 0 \\
   &   &   & 1 & 0 \\
   &   &   &   & 0
\end{smallmatrix} \quad
\begin{smallmatrix}
 2 & 0 & 1 & 1 & 0 \\
   & 2 & 0 & 0 & 0 \\
   &   & 1 & 1 & 0 \\
   &   &   & 0 & 0 \\
   &   &   &   & 0
\end{smallmatrix} \quad
\begin{smallmatrix}
 2 & 0 & 1 & 0 & 0 \\
   & 2 & 0 & 1 & 0 \\
   &   & 1 & 0 & 0 \\
   &   &   & 1 & 0 \\
   &   &   &   & 0
\end{smallmatrix} \quad
\begin{smallmatrix}
 2 & 1 & 2 & 0 & 0 \\
   & 3 & 1 & 1 & 0 \\
   &   & 1 & 1 & 0 \\
   &   &   & 0 & 0 \\
   &   &   &   & 0
\end{smallmatrix} \quad
\begin{smallmatrix}
 2 & 1 & 2 & 1 & 0 \\
   & 3 & 1 & 0 & 0 \\
   &   & 1 & 0 & 0 \\
   &   &   & 1 & 0 \\
   &   &   &   & 0
\end{smallmatrix}
$$

$$
\begin{smallmatrix}
 3 & 0 & 1 & 1 & 0 \\
   & 1 & 2 & 0 & 0 \\
   &   & 1 & 1 & 0 \\
   &   &   & 0 & 0 \\
   &   &   &   & 0
\end{smallmatrix} \quad
\begin{smallmatrix}
 3 & 0 & 1 & 0 & 0 \\
   & 1 & 2 & 1 & 0 \\
   &   & 1 & 0 & 0 \\
   &   &   & 1 & 0 \\
   &   &   &   & 0
\end{smallmatrix} \quad
\begin{smallmatrix}
 3 & 2 & 0 & 0 & 0 \\
   & 3 & 1 & 1 & 0 \\
   &   & 1 & 1 & 0 \\
   &   &   & 0 & 0 \\
   &   &   &   & 0
\end{smallmatrix} \quad
\begin{smallmatrix}
 3 & 2 & 0 & 1 & 0 \\
   & 3 & 1 & 0 & 0 \\
   &   & 1 & 0 & 0 \\
   &   &   & 1 & 0 \\
   &   &   &   & 0
\end{smallmatrix} \quad
\begin{smallmatrix}
 4 & 1 & 1 & 1 & 0 \\
   & 1 & 1 & 1 & 0 \\
   &   & 0 & 0 & 0 \\
   &   &   & 0 & 0 \\
   &   &   &   & 0
\end{smallmatrix} \quad
\begin{smallmatrix}
 4 & 1 & 1 & 0 & 0 \\
   & 1 & 1 & 0 & 0 \\
   &   & 0 & 1 & 0 \\
   &   &   & 1 & 0 \\
   &   &   &   & 0
\end{smallmatrix} \quad
\begin{smallmatrix}
 4 & 2 & 1 & 1 & 0 \\
   & 2 & 1 & 1 & 0 \\
   &   & 2 & 0 & 0 \\
   &   &   & 0 & 0 \\
   &   &   &   & 0
\end{smallmatrix} \quad
\begin{smallmatrix}
 4 & 2 & 1 & 0 & 0 \\
   & 2 & 1 & 0 & 0 \\
   &   & 2 & 1 & 0 \\
   &   &   & 1 & 0 \\
   &   &   &   & 0
\end{smallmatrix}
$$

From this data, one has to guess two sets of conditions on residue tables so that precisely one of them is satisfied for all the tables above, and neither of them is satisfied for any of the other $96$ tables. It is easy to see that the following sets of conditions work:

\begin{itemize}
\item the condition
$$r_{14} \neq r_{34} \mbox{ and } r_{14} \neq r_{44} \mbox{ and } r_{24} \neq r_{34} \mbox{ and } r_{24} \neq r_{44}$$
is satisfied for the first four and the last four of the residue tables above, but not for others; we declare them the $k$-LR conditions for $2211$
 \item the condition
$$r_{14} \neq r_{24} \mbox{ and } r_{34} \neq r_{44}$$
is satisfied for the residue tables $5$--$20$ above, but not for the others; we declare them the $k$-LR conditions for $2121$.
\end{itemize}

These exact same conditions work for $k = 6, 7, 8, 9$ as well, as a straightforward (and time consuming) computer check shows.

\medskip

The author was able to obtain (and check for a large number of cases) such conditions for LR arrays with an arbitrary number of $1$'s, at most two $2$'s, and no dots before or in between\footnote{In a previous version of this manuscript, these conditions were written in a much more complicated manner in terms of the first $n-1$ columns of the residue table; it was Mar\' ia Elena Pinto who realized that the conditions are much simpler if we just observe the $n$-th column.}.

\medskip 

For no $2$'s (\textit{i.e.}, for the LR array $\p A = 11\ldots 1$), there are no conditions to satisfy; this is consistent with Theorem \ref{whs}. For one $2$, \textit{i.e.}\ for an LR array $\p A = 1\ldots 1 2 1 \ldots 1$ (with $2$ in position $I$ and total length $n$), the condition is $r_{I,n} \neq r_{j,n}$ for $j = I+1,\ldots,n$.

\begin{example}
 Let us say that we want to compute the coefficient of $s_{65444433222}^{(8)}$ in $s_{5444433222}^{(8)} s_{51}^{(8)}$. Note that $65444433222 = 5444433222^{\set{1,2,3,4,5,6}}$, so we are interested in LR arrays of shape $111111$ and weight $51$. These arrays, and the conditions, are given by the following:
 \begin{description}
  \item[211111] $r_{16} \neq r_{26},r_{36},r_{46},r_{56},r_{66}$
  \item[121111] $r_{26} \neq r_{36},r_{46},r_{56},r_{66}$
  \item[112111] $r_{36} \neq r_{46},r_{56},r_{66}$
  \item[111211] $r_{46} \neq r_{56},r_{66}$
  \item[111121] $r_{56} \neq r_{66}$
 \end{description}
 The residue table of $5444433222$ (for $k = 8$) is
 $$ \begin{smallmatrix}
  0 & 3 & 5 & 4 & 1 & 1 & 1 & 0 \\
   & 3 & 5 & 4 & 1 & 1 & 1 & 0 \\
   &   & 2 & 1 & 2 & 2 & 0 & 0 \\
   &   &   & 4 & 1 & 1 & 1 & 0 \\
   &   &   &   & 1 & 1 & 1 & 0 \\
   &   &   &   &   & 0 & 0 & 0 \\
   &   &   &   &   &   & 0 & 0 \\
   &   &   &   &   &   &   & 0 \\
  \end{smallmatrix}$$
  and it is easy to check that the conditions for $112111$ and $111121$ are satisfied, while the conditions for $211111$, $121111$, and $111211$ are not (because of, for example, $r_{16} = r_{26}$, $r_{26} = r_{46}$, and $r_{46} = r_{56}$, respectively), so $c_{5444433222,51}^{(8),65444433222} = 2$, as can be checked independently.
\end{example}

Say that we have an LR array of shape $11\ldots 1$ ($n$ ones) and weight $1^{n-2} 2^2$, in other words, we have $\p A = 1\ldots 1 2 1 \ldots 1 2 1 \ldots 1$, with $2$'s in positions $I$ and $J$. If $J \leq n - 2$, the conditions are $r_{In} \neq r_{jn}$ for $j = I+1,\ldots,J-1,J+1,\ldots,n$, $r_{Jn} \neq r_{jn}$ for $j = J+1,\ldots,n$. On the other hand, if $J = n - 1$, the condition are $r_{In} \neq r_{jn}$ for $j = I+1,\ldots,n-2$, $r_{n-1,n} \neq r_{nn}$.

\medskip

Example \ref{exintro} in Section \ref{intro} illustrates this for $n = 6$. The conditions \textbf{C1}--\textbf{C9} correspond to LR arrays $221111$, $212111$, $211211$, $211121$, $122111$, $121211$, $121121$, $112211$, $112121$, respectively.

\medskip

These conditions can be described in a unified manner: for an LR array $\p A$ consisting of $1$'s and at most two $2$'s, we get the condition $r_{in} \neq r_{jn}$ if and only if $\p A$ has $2$ in position $i$, $1$ in position $j$, unless $i \leq n-2$, $j = n$, and $\p A$ has $2$ in position $n-1$.

\medskip

One would hope for a general pattern: for an LR array $\p A$ consisting of $1$'s and $2$'s, we get the condition $r_{in} \neq r_{jn}$ if and only if $\p A$ has $2$ in position $i$, $1$ in position $j$, unless the number of $1$'s and $2$'s from position $l$ onwards is the same for some $l$, $i < l < j$. Another way to state this is as follows. If we interpret $1$ as an NE step and $2$ as a SE step and read the LR array from right to left, we get a path that starts on the $x$-axis never goes below it. Such a path can be divided into primitive Dyck paths, and another (possibly empty) primitive path that never returns to the $x$-axis. Every $2$ and $1$ in positions $i$ and $j$, $i < j$, that belong to the same primitive path give a condition $r_{in} \neq r_{jn}$.

\medskip

Unfortunately, such conditions fails for $\mu = 33$, \textit{i.e.} for LR arrays with three $1$'s and three $2$'s. They work for $k \leq 7$, but hold only for approximately $99.5\%$ of irreducible $8$-bounded partitions, $99\%$ of irreducible $9$-bounded partitions, $98.7\%$ of irreducible $10$-bounded partitions, etc.

\medskip

Perhaps a variant of this could work. After such conditions are found, the following would have to be resolved before a full conjecture on $k$-Littlewood-Richardson coefficients can be formed:
\begin{itemize}
 \item allow $\mu$ to have length $3$ or more, \textit{i.e.}\ allow $\p A$ to contain $1$'s, $2$'s, $3$'s etc.;
 \item figure out what happens when ``dots'' are added, \textit{i.e.}\ when we are allowed to have $A_i = \emptyset$ and $A_i = \set j$; 
 \item allow $\p A$ to be arbitrary, \textit{i.e.}\ allow $A_i$'s to contain any number of elements.
\end{itemize}

Judging from the work done so far, and perhaps counterintuitively, the author believes these three tasks will be easier than the figuring out the general case with $1$'s and $2$'s only.

\section{Proofs} \label{proofs}

This section contains two proofs that are quite technical and not necessary for the understanding of the main results of this paper.

\begin{proof}[Proof of Theorem \ref{size}]
 Think of the cells in $\mf c^{(k)}(\lambda)$ with hook-length $> k$ as empty. Our goal is to count the empty cells of $\mf c^{(k)}(\lambda)$.\\
 For $0 \leq J \leq k$, again write $\lambda_{(J)}$ for $\langle 1^{m_1(\lambda)},\ldots,J^{m_J(\lambda)} \rangle$. We prove that
 $$|\mf c^{(k)}(\lambda_{(J)})| - |\lambda_{(J)}| = \!\!\!\! \sum_{1 \leq i \leq j \leq J}\!\!\!\! r_{ij} q_{ij}+ \!\!\!\!\sum_{1 \leq i \leq j < h \leq J}\!\!\!\! r_{hh}q_{ij} + \!\!\!\!\sum_{1 \leq i \leq j \leq h \leq J}\!\!\!\! (k+1-h)q_{hh}q_{ij} - \!\!\sum_{1 \leq i \leq J	}\!\! i (k+1-i) \left(\begin{smallmatrix}q_{ii}+1 \\ 2\end{smallmatrix}\right)$$
 by induction on $J$, the theorem is this statement for $J = k$. Denote the sum on the right by $A$.\\
 The equality is obvious for $J = 0$, assume that it holds for $J - 1$. We prove the statement by induction on $m = m_J(\lambda_{(J)})$.\\
 If $m = 0$, $\lambda_{(J)} = \lambda_{(J-1)}$. By the (outer) induction hypothesis,
 \begin{align*}
  |\mf c^{(k)}(\lambda_{(J)})| - |\lambda_{(J)}| &= |\mf c^{(k)}(\lambda_{(J-1)})| - |\lambda_{(J-1)}| \\ 
&= \!\!\!\!\! \sum_{1 \leq i \leq j < J}\!\!\!\!\! r_{ij} q_{ij}+ \!\!\!\!\!\sum_{1 \leq i \leq j < h < J}\!\!\!\!\! r_{hh}q_{ij} + \!\!\!\!\!\sum_{1 \leq i \leq j \leq h < J}\!\!\!\!\! (k+1-h)q_{hh}q_{ij} - \!\!\!\sum_{1 \leq i < J}\! i (k+1-i) \left(\begin{smallmatrix}q_{ii}+1 \\ 2\end{smallmatrix}\right) 
 \end{align*}
 To complete the base of (inner) induction, we need to see that none of these four sums change when we replace $< J$ by $\leq J$. Since $r_{JJ} = q_{JJ} = 0$, this is clear for the second, third, and fourth sum. By definition of residue and quotient tables, $r_{iJ} = r_{i,J-1} \md (k+1-J)$ and $q_{iJ} = r_{i,J-1} \dv (k+1-J)$. Since $r_{i,J-1} \leq k+1-J$, $q_{iJ} \neq 0$ implies $r_{iJ} = 0$. Therefore $\sum_{1 \leq i \leq j < J} r_{ij} q_{ij} = \sum_{1 \leq i \leq j \leq J} r_{ij} q_{ij}$ as well.\\
 Now assume that $m > 0$ and that the statement holds for the partition $\bar \lambda$ obtained from $\lambda$ by removing one copy of $J$; denote its residue and quotient tables by $\bar R$ and $\bar Q$, respectively. By (inner) induction,
 $$|\mf c^{(k)}(\bar \lambda_{(J)})| - |\bar \lambda_{(J)}| = \!\!\!\! \sum_{1 \leq i \leq j \leq J}\!\!\!\! \bar r_{ij} \bar q_{ij}+ \!\!\!\!\sum_{1 \leq i \leq j < h \leq J}\!\!\!\! \bar r_{hh}\bar q_{ij} + \!\!\!\!\sum_{1 \leq i \leq j \leq h \leq J}\!\!\!\! (k+1-h)\bar q_{hh}\bar q_{ij} - \!\!\sum_{1 \leq i \leq J}\!\! i (k+1-i) \left(\begin{smallmatrix}\bar q_{ii}+1 \\ 2\end{smallmatrix}\right).$$
 Denote the sum on the right by $B$. We also know that $\mf c^{(k)}(\bar \lambda_{(J)})$ has $C = \sum_{1 \leq i \leq j \leq J} \bar q_{ij}$ columns of length $> k - J$, and the lengths of the other columns are precisely the non-zero entries of column $J$ of $\bar R$. When adding a new row of length $J$ to the core $\mf c^{(k)}(\bar \lambda_{(J)})$, we push it to the right until its left-most cell is positioned above the first column of $\mf c^{(k)}(\bar \lambda_{(J)})$ of length $\leq k - J$. In other words, we have increased the number of empty cells by $C$. Therefore it remains to prove that $B + C = A$. Since $r_{ij} = \bar r_{ij}$ and $q_{ij} = \bar q_{ij}$ for $1 \leq i \leq j < J$, this is equivalent to
 \begin{align*}
 &\phantom{=}\sum_{i=1}^J \bar r_{iJ} \bar q_{iJ} + \bar r_{JJ} \sum_{1 \leq i \leq j < J} q_{ij} + (k+1-J)\bar q_{JJ} \sum_{1 \leq i \leq j \leq J} \bar q_{ij} - J(k+1-J) \smallbinom{\bar q_{JJ}+1}2 + \sum_{1 \leq i \leq j \leq J} \bar q_{ij} \\
 &= \sum_{i=1}^J r_{iJ} q_{iJ} + r_{JJ} \sum_{1 \leq i \leq j < J} q_{ij} + (k+1-J)q_{JJ} \sum_{1 \leq i \leq j \leq J} q_{ij} - J(k+1-J) \smallbinom{q_{JJ}+1}2 \qquad \qquad  (*)
 \end{align*}
 Denote by $S$ the set of $i$, $1 \leq i \leq J$, for which $r_{iJ} = 0$.\\
 (a) First assume that $J \notin S$. For all $i \in S$, we have $\bar r_{iJ} = k - J$, $\bar q_{iJ} + 1 = q_{iJ} = q_{JJ} + 1$, and for all $i \notin S$, we have $\bar r_{iJ} + 1 = r_{iJ}$, $\bar q_{iJ} = q_{iJ}$. The sums on the left-hand side of $(*)$ equal
 \begin{align*} 
\sum_{i=1}^J \bar r_{iJ} \bar q_{iJ} &= \!\sum_{i \in S} (k-J) (q_{iJ} - 1) + \! \sum_{i \notin S} (r_{iJ}-1) q_{iJ} \\
 & = \left[(k-J) q_{JJ} |S| - \sum_{i \notin S} q_{iJ}\right] \! +\!\sum_{i =1}^J r_{iJ} q_{iJ},\\
 \bar r_{JJ} \sum_{1 \leq i \leq j < J} q_{ij} &= \left[ - \sum_{1 \leq i \leq j < J} q_{ij}\right] + r_{JJ} \sum_{1 \leq i \leq j < J} q_{ij}, \\
 (k+1-J)\bar q_{JJ} \sum_{1 \leq i \leq j \leq J} \bar q_{ij} &= (k+1-J)q_{JJ} \left( \sum_{1 \leq i \leq j \leq J} q_{ij} - |S| \right) \\
 &= \bigg[- (k+1-J)q_{JJ}|S|\bigg] + (k+1-J)q_{JJ} \sum_{1 \leq i \leq j \leq J} q_{ij}, \\
 - J(k+1-J)\smallbinom{\bar q_{JJ}+1}2 &= - J(k+1-J)\smallbinom{q_{JJ}+1}2,\\
 \sum_{1 \leq i \leq j \leq J} \!\bar q_{ij} &= \!\left[\!\sum_{1 \leq i \leq j < J} q_{ij} + |S| q_{JJ} + \sum_{i \notin S} q_{iJ}\right],
 \end{align*}
 where the sums that do not appear on the right-hand side of $(*)$ are enclosed in brackets. Sum the left-hand sides of the last five equalities and subtract the right-hand side of $(*)$:
$$(k-J) q_{JJ} |S| - \sum_{i \notin S} q_{iJ} - \sum_{1 \leq i \leq j < J} q_{ij}  - (k+1-J)q_{JJ}|S| + \sum_{1 \leq i \leq j < J} q_{ij} + |S| q_{JJ} + \sum_{i \notin S} q_{iJ}=0$$
 (b) The second case is when $J \in S$; then $\bar r_{iJ} = k - J$, $r_{iJ} = 0$, $\bar q_{iJ} + 1 = q_{iJ}$ for all $i \in S$, and for all $i \notin S$, we have $\bar r_{iJ} + 1 = r_{iJ}$, $\bar q_{iJ} = q_{iJ} = q_{JJ}$. The five sums are now
\begin{align*}
\sum_{i=1}^J \bar r_{iJ} \bar q_{iJ} &= \!\sum_{i \in S} (k-J) (q_{iJ} - 1) + \! \sum_{i \notin S} (r_{iJ}-1) q_{iJ} \\
 & = \left[(k-J) \sum_{i \in S} q_{iJ} - (k-J)|S| - \sum_{i \notin S} q_{iJ}\right] \! +\!\sum_{i =1}^J r_{iJ} q_{iJ},\\
  \bar r_{JJ} \sum_{1 \leq i \leq j < J} q_{ij} &=  \left[(k-J) \sum_{1 \leq i \leq j < J} q_{ij}\right] + r_{JJ} \sum_{1 \leq i \leq j < J} q_{ij}, \\
 (k+1-J)\bar q_{JJ} \!\!\!\sum_{1 \leq i \leq j \leq J}\!\!\! \bar q_{ij} &= (k+1-J)(q_{JJ}-1) \left( \sum_{1 \leq i \leq j \leq J} \!\!\!q_{ij} - |S| \right) \\
 &= \left[- (k+1-J)\left((q_{JJ}-1)|S| + \!\!\!\sum_{1 \leq i \leq j \leq J} \!\!\!q_{ij} \right)\right] + (k+1-J)q_{JJ} \!\!\!\sum_{1 \leq i \leq j \leq J}\!\!\! q_{ij}, \\
 - J(k+1-J)\smallbinom{\bar q_{JJ}+1}2 &= \bigg[ J(k+1-J)q_{JJ} \bigg] - J(k+1-J)\smallbinom{q_{JJ}+1}2,\\
\sum_{1 \leq i \leq j \leq J} \!\bar q_{ij} &= \left[\sum_{1 \leq i \leq j < J} q_{ij} + \sum_{i \in S} q_{iJ} - |S| + \sum_{i \notin S} q_{iJ}\right],
\end{align*}
Summing the left-hand sides of these equations and subtracting the right-hand side of $(*)$, we get
\begin{align*}
& \left[(k-J) \sum_{i \in S} q_{iJ} - \cancel{(k-J)|S|} - \bcancel{\sum_{i \notin S} q_{iJ}}\right] + \left[(k-J) \sum_{1 \leq i \leq j < J} q_{ij}\right] \\
&+ \left[- (k+1-J)\left((q_{JJ}-\cancel{1})|S| + \sum_{1 \leq i \leq j \leq J} q_{ij} \right)\right] + \bigg[ J(k+1-J)q_{JJ} \bigg] \\
& + \left[\sum_{1 \leq i \leq j < J} q_{ij} + \sum_{i \in S} q_{iJ} - \cancel{|S|} + \bcancel{\sum_{i \notin S} q_{iJ}}\right] \\
& = \cancel{(k-J) \sum_{i \in S} q_{iJ}} + \bcancel{(k-J) \sum_{1 \leq i \leq j < J} q_{ij}} - \xcancel{(k+1-J) q_{JJ}|S|} \\
& - (k+1-J)  \left( \bcancel{\sum_{1 \leq i \leq j < J}  q_{ij}} + \cancel{\sum_{i \in S} q_{iJ}} + \xcancel{(J - |S|) q_{JJ}}\right) + \xcancel{J(k+1-J)q_{JJ}} + \bcancel{\sum_{1 \leq i \leq j < J} q_{ij}} +  \cancel{\sum_{i \in S} q_{iJ}} \\
& = 0
\end{align*}
This completes the proof.
\end{proof}

\begin{proof}[Proof of Lemma \ref{lemma}]
 The proof of (a) is a careful examination of all cases and is by induction on $j \geq i$ for $i$ fixed. First recall that for $j = 1,\ldots,k$,

 $$m'_j = \left\{ 
\begin{array}{ccl}
 m_j + 1 & \colon j \in S, j+1 \notin S \\
 m_j - 1 & \colon j \notin S, j+1 \in S \\
 m_j & \colon j \in S, j+1 \in S \\
 m_j & \colon j \notin S, j+1 \notin S
\end{array}
\right.$$
 \begin{description}
  \item[Case I] $i \notin S$:\\
  We have $r_{ii} = m_i \md (k+1-i)$, $r'_{ii} = m'_i \md (k+1-i)$, $q_{ii} = m_i \dv (k+1-i)$, $q'_{ii} = m'_i \dv (k+1-i)$. If $i+1 \in S$, then $m_i' = m_i-1$ and $r_{ii} > 0$, so $r'_{ii} = (m_i-1) \md (k+1-i) = r_{ii}-1$ and $q'_{ii} = (m_i-1) \dv (k+1-i) = q_{ii}$. If $i+1 \notin S$, then $m_i' = m_i$ and $r'_{ii} = m_i \md (k+1-i) = r_{ii}$ and $q'_{ii} = m_i \dv (k+1-i) = q_{ii}$. This proves the case $i = j$. For $j > i$, we have four cases:
  \begin{description}
   \item[Case I.1] $j \in S$, $j+1 \in S$. In this case we have $r'_{i,j-1} = r_{i,j-1}-1$ by induction, $m_j' = m_j$. Therefore $r'_{ij} = (r'_{i,j-1} + m_j) \md (k+1-j) = (r_{i,j-1} - 1 + m_j) \md (k+1-j) = r_{ij} - 1$ because both $r_{i,j-1}$ and $r_{ij}$ are positive.
   \item[Case I.2] $j \notin S$, $j+1 \in S$. In this case we have $r'_{i,j-1} = r_{i,j-1}$ by induction, $m_j' = m_j-1$. Therefore $r'_{ij} = (r'_{i,j-1} + m_j) \md (k+1-j) = (r_{i,j-1} + m_j -1) \md (k+1-j) = r_{ij} - 1$ because $r_{ij}$ is positive.
   \item[Case I.3] $j \in S$, $j+1 \notin S$. In this case we have $r'_{i,j-1} = r_{i,j-1}-1$ by induction, $m_j' = m_j+1$. Therefore $r'_{ij} = (r'_{i,j-1} + m_j) \md (k+1-j) = (r_{i,j-1} - 1 + m_j + 1) \md (k+1-j) = r_{ij}$ because $r_{i,j-1}$ is positive.
   \item[Case I.4] $j \notin S$, $j+1 \notin S$. In this case we have $r'_{i,j-1} = r_{i,j-1}$ by induction, $m_j' = m_j$. Therefore $r'_{ij} = (r'_{i,j-1} + m_j) \md (k+1-j) = (r_{i,j-1} + m_j) \md (k+1-j) = r_{ij}$.
  \end{description}
  In all four cases, $q'_{ij} = q_{ij}$.
  \item[Case II] $i \in S$:\\
  If $i + 1 \in S$, then in particular $i < h(i)$, so we have to prove that $r'_{ii} = r_{ii}$ and $q'_{ii} = q_{ii}$. Since $m'_i = m_i$ and $r_{ii} = m_i \md (k+1-i)$, $r'_{ii} = m'_i \md (k+1-i)$, $q_{ii} = m_i \dv (k+1-i)$, $q'_{ii} = m'_i \dv (k+1-i)$, this is obvious.\\
  If $i + 1 \notin S$ and $i < h(i)$, then we have $r_{ii} < k - i$ (otherwise $h(i)$ would equal $i$). Then $m'_i = m_i + 1$ and $r'_{ii} = (m_i + 1) \md (k+1-i) = r_{ii}+1$, $q'_{ii} = (m_i + 1) \dv (k+1 - i) = q_{ii}$. If $i+1 \notin S$ and $i = h(i)$, then $r_{ii} = k - i$ and so $r'_{ii} = (m_i + 1) \md (k+1-i) = 0$ and $q'_{ii} = (m_i + 1) \dv (k+1-i) = q_{ii} + 1$. This completes the case $i = j$. For $j > i$, we have five cases, each of which has two or three subcases. But first let us prove the last statement of (a). By definition of $h(i)$, we have $r_{i,h(i)} = k - h(i)$, and since $m_{h(i)}$ is either $m_{h(i)}$ or $m_{h(i)+1}$, it follows that $r'_{i,h(i)}$ is either $k-h(i)$ or $0$. In either case, $r_{i,h(i)+1} = r_{h(i)+1,h(i)+1}$ and $r'_{i,h(i)+1} = r'_{h(i)+1,h(i)+1}$, $r_{i,h(i)+2} = r_{h(i)+1,h(i)+2}$ and $r'_{i,h(i)+2} = r'_{h(i)+1,h(i)+2}$, \textit{etc.}
  \begin{description}
   \item[Case II.1] $j+1 \in S$, $j < h(i)$: We need to prove $r'_{ij} = r_{ij}$ and $q'_{ij} = q_{ij}$:
   \begin{description}
    \item[Case II.1.a] $j \in S$: Since $j-1 < h(i)$, we have $r'_{i,j-1} = r_{i,j-1}$ by induction and $m'_j = m_j$ because $j,j+1 \in S$, so $r'_{ij} = (r'_{i,j-1} + m'_j) \md (k+1-j) = (r_{i,j-1} + m_j) \md (k+1-j) = r_{ij}$ and $q'_{ij} = (r'_{i,j-1} + m'_j) \dv (k+1-j) = (r_{i,j-1} + m_j) \dv (k+1-j) = q_{ij}$.
    \item[Case II.1.b] $j \notin S$: We have $r'_{i,j-1} = r_{i,j-1}+1$ and $m'_j = m_j -1$, so $r'_{ij} = r_{ij}$ and $q'_{ij} = q_{ij}$.
   \end{description}
   \item[Case II.2] $j+1 \in S$, $j > h(i)$: We need to prove that $r'_{ij} = r_{ij}-1$, and that $q'_{ij} = q_{ij}$ if $j > h(i) +1$ and $q'_{ij} = q_{ij}-1$ if $j = h(i) + 1$:
   \begin{description}
    \item[Case II.2.a] $j \in S$: In this case $j - 1 > h(i)$, so $r'_{i,j-1} = r_{i,j-1} - 1$ and $m'_j = m_j$; therefore $r'_{ij} = r_{ij} - 1$ and $q'_{ij} = q_{ij}$.
    \item[Case II.2.b] $j \notin S$, $j - 1 > h(i)$: Now $r'_{i,j-1} = r_{i,j-1}$ and $m'_j = m_j - 1$; we know that $r'_{ij} = r'_{h(i)+1,j}$ and $r_{ij} = r_{h(i)+1,j} > 0$ so $r'_{ij} = r_{ij} - 1$ and $q'_{ij} = q_{ij}$.
    \item[Case II.2.c] $j \notin S$, $j - 1 = h(i)$: Now $r_{i,j-1} = k+1-j$, $r'_{i,j-1} = 0$, $m'_j = m_j-1$; we know that $r'_{ij} = r'_{h(i)+1,j}$ and $r_{ij} = r_{h(i)+1,j} > 0$ and therefore $r'_{ij} = r_{ij} - 1$, $q'_{ij} = q_{ij}-1$.
   \end{description}
   \item[Case II.3] $j+1 \notin S$, $j < h(i)$: We need to prove $r'_{ij} = r_{ij}+1$ and $q'_{ij} = q_{ij}$:  
   \begin{description}
    \item[Case II.3.a] $j \in S$: In this case, $r'_{i,j-1} = r_{i,j-1}$ and $m'_j = m_j + 1$. Furthermore, since $j \neq h(i)$, $r_{i,j-1} < k + 1 - j$. This implies that $r'_{ij} = r_{ij} + 1$ and $q'_{ij} = q_{ij}$.
    \item[Case II.3.b] $j \notin S$: Now $r'_{i,j-1} = r_{i,j-1} + 1$, $r_{i,j-1} < k + 1 - j$, and $m'_j = m_j$. This implies that $r'_{ij} = r_{ij} + 1$ and $q'_{ij} = q_{ij}$.
   \end{description}
   \item[Case II.4] $j+1 \notin S$, $j = h(i)$: We have $r_{ij} = k - i$, we need to prove $r'_{ij} = 0$ and $q'_{ij} = q_{ij} + 1$:
   \begin{description}
    \item[Case II.4.a] $j \in S$: We have $j-1<h(i)$, so $r'_{i,j-1} = r_{i,j-1}$ and $m'_j = m_j + 1$. Therefore $r'_{ij} = (r_{i,j-1} + m_j + 1) \md (k+1-j) = 0$ and $q'_{ij} = (r_{i,j-1} + m_j + 1) \dv (k+1-j) = q_{ij}+1$.
    \item[Case II.4.b] $j \notin S$: Again, $j-1<h(i)$, so $r'_{i,j-1} = r_{i,j-1} + 1$, $r_{i,j-1} < k + 1 - j$, and $m'_j = m_j$; therefore $r'_{ij} = (r_{i,j-1} + 1 + m_j) \md (k+1-j) = 0$ and $q'_{ij} = (r_{i,j-1} +1 + m_j) \dv (k+1-j) = q_{ij}+1$.
   \end{description}
   \item[Case II.5] $j+1 \notin S$, $j > h(i)$: We need to prove that $r'_{ij} = r_{ij}$, and that $q'_{ij} = q_{ij}$ if $j > h(i) +1$ and $q'_{ij} = q_{ij}-1$ if $j = h(i) + 1$:
   \begin{description}
    \item[Case II.5.a] $j \in S$: We have $j - 1 > h(i)$, so $r'_{i,j-1} = r_{i,j-1} - 1\geq 0$ and $m'_j = m_j + 1$; therefore $r'_{ij} = r_{ij} $ and $q'_{ij} = q_{ij}$.
    \item[Case II.5.b] $j \notin S$, $j - 1 > h(i)$: Now $r'_{i,j-1} = r_{i,j-1}$ and $m'_j = m_j$, so $r'_{ij} = r_{ij}$ and $q'_{ij} = q_{ij}$.
    \item[Case II.5.c] $j \notin S$, $j - 1 = h(i)$: Now $r_{i,j-1} = k+1-j$, $r'_{i,j-1}=0$, $m'_j = m_j$ and $r'_{ij} = m_j \md (k+1-j) = (r_{i,j-1} + m_j) \md (k+1-j) = r_{ij}$ and $q'_{ij} = q_{ij}-1$.
   \end{description}
  \end{description}
 \end{description}
To prove (b), assume that $J$ is the smallest $j$ for which $j \in S$, $r_{i,j-1} = 0$ for some $i \notin S$, $i < j$. The computations from (a) still hold for $j < J$; in particular, in each line $i < J$ up to $J-1$, we have $q'_{ij} = q_{ij}$ unless $i \in S$ and $j = h(i)$ or $j = h(i)+1$, in which case $q'_{ij}=q_{ij}+1$ or $q'_{ij}=q_{ij}-1$. Furthermore, $h(i)+1 \notin S$, so $J \neq h(i) + 1$. That implies that the total number of parts $\geq k + 3 - J$ in $\lambda^{(k)}$ and $(\lambda^S)^{(k)}$ is the same. There is a unique $I$, $1 \leq I < J$, so that $I \notin S$, $I+1,\ldots,J \in S$. Then $m_j(\lambda^S) = m_j(\lambda)$ for $j = I+1,\ldots,J-1$, and $m_I(\lambda^S) = m_I(\lambda) - 1$ (if $I > 0$). Now if $i \notin S$, $r_{i,J-1} = 0$, then by the minimality of $J$, we must have $r_{ij} > 0$ for $j = I,\ldots,J-2$. Therefore $r'_{ij} = r_{ij} - 1$ and $q'_{ij} = q_{ij}$ for $j = I,\ldots,J-2$, and necessarily $r'_{i,J-1} = k +1 - J$ and $q'_{i,J-1} = q_{i,J-1}-1$. That means that there are fewer copies of $k+2-J$ in $(\lambda^S)^{(k)}$ than in $\lambda^{(k)}$, and this implies $\lambda^{(k)} \not\subseteq (\lambda^S)^{(k)}$.
\end{proof}

\section*{Acknowledgments}
\label{sec:ack}
Many thanks to Susanna Fishel for suggesting a project in $k$-Schur function theory that has led to the definition of residue and quotient tables; to Luc Lapointe for the much needed encouragement and many helpful suggestions; to Mar\' ia Elena Pinto for helping me realize that the conditions from Subsection \ref{toward} can be phrased in such a simple manner; to Jennifer Morse for suggesting looking into products with almost-$k$-rectangles; and to Christopher Hanusa for suggesting finding the size of the core corresponding to a $k$-bounded partition. Kudos also to Jennifer Morse, Anne Schilling, and Mike Zabrocki for the highly recommended introduction to $k$-Schur functions \cite[\S 2]{llmssz}.

\bibliographystyle{alpha}
\bibliography{bib}
\label{sec:biblio}

\end{document}